\documentclass[a4paper]{amsart}

\usepackage{amsmath,amssymb}
\usepackage{amsthm}

\newtheorem{theorem}{Theorem}[section]
\newtheorem{proposition}[theorem]{Proposition}
\newtheorem{corollary}[theorem]{Corollary}
\newtheorem{lemma}[theorem]{Lemma}
\newtheorem{algorithm}[theorem]{Algorithm}

\newtheorem{preremark}[theorem]{Remark}
\newtheorem{predefinition}[theorem]{Definition}
\newtheorem{preexample}[theorem]{Example}
\newtheorem{prenotation}[theorem]{Notation}
\newtheorem{preconjecture}[theorem]{Conjecture}

\newenvironment{remark}{\begin{preremark}\rm}{\end{preremark}}
\newenvironment{definition}{\begin{predefinition}\rm}
{\end{predefinition}}
\newenvironment{example}{\begin{preexample}\rm}{\end{preexample}}

\newenvironment{notation}{\begin{prenotation}\rm}{\end{prenotation}}

\def\NF{\mathop{\rm NF}\nolimits}

\def\LF{\mathop{\rm LF}\nolimits}
\def\DF{\mathop{\rm DF}\nolimits}
\def\HF{\mathop{\rm HF}\nolimits}

\def\ri{\mathop{\rm ri}\nolimits}

\def\Mat{\mathop{\rm Mat}\nolimits}
\def\Hom{\mathop{\rm Hom}\nolimits}
\def\rank{\mathop{\rm rank}\nolimits}

\def\Supp{\mathop{\rm Supp}\nolimits}
\def\Spec{\mathop{\rm Spec}\nolimits}

\renewcommand{\hom}{{\mathop{\rm hom}}}

\def\Soc{\mathop{\rm Soc}\nolimits}
\def\Ann{\mathop{\rm Ann}\nolimits}

\def\Rad{\mathop{\rm Rad}\nolimits}

\def\Ker{\mathop{\rm Ker}\nolimits}

\def\id{\mathop{\rm id}\nolimits}
\def\Col{\mathop{\rm Col}\nolimits}
\def\gr{\mathop{\rm gr}\nolimits}

\def\ord{\mathop{\rm ord}\nolimits}

\def\sepdeg{\mathop{\rm sepdeg}\nolimits}

\def\m{{\mathfrak{m}}}

\def\M{\mathfrak{M}}

\def\q{{\mathfrak{q}}}
\def\Q{{\mathfrak{Q}}}

\def\Fbar{{\mathcal{F}}}
\def\grFR{{\gr_{\mathcal{F}}(R)}}

\DeclareSymbolFont{newfont}{OML}{cmm}{m}{it}
\DeclareMathSymbol{\Varrho}{3}{newfont}{37}
\def\rho{{\mathop{\Varrho}\,}}

\let\epsilon=\varepsilon

\def\phi{{\varphi}}
\let\Psi=\varPsi
\let\Phi=\varPhi
\let\theta=\vartheta

\def\tr{^{\,\rm tr}}
\def\tfrac #1#2{{\textstyle\frac{#1}{#2}}}

\def\cocoa{\mbox{\rm
  C\kern-.13em o\kern-.07 em C\kern-.13em o\kern-.15em A}}
\def\apcocoa{\mbox{\rm
A\kern-0.13em p\kern -0.07em C\kern-.13em o\kern-.07 em C\kern-.13em
o\kern-.15em A}}

\begin{document}

\title{On the Cayley-Bacharach Property}

\author{Martin Kreuzer}
\address{Fakult\"at f\"ur Informatik und Mathematik, Universit\"at
Passau,
D-94030 Passau, Germany}
\email{Martin.Kreuzer@uni-passau.de}

\author{Le Ngoc Long}
\address{Fakult\"at f\"ur Informatik und Mathematik, Universit\"at
Passau,
D-94030 Passau, Germany}
\email{nglong16633@gmail.com}

\author{Lorenzo Robbiano}
\address{Dipartimento di Matematica, Universit\`a di Genova,
Via Dodecaneso 35,
I-16146 Genova, Italy}
\email{lorobbiano@gmail.com}

\date{\today}

\begin{abstract}
The Cayley-Bacharach property, which has been classically stated
as a property of a finite set of points in an affine or projective
space, is extended to arbitrary 0-dimensional affine algebras
over arbitrary base fields. We present characterizations and
explicit algorithms for checking the Cayley-Bacharach property
directly, via the canonical module, and in combination with the
property of being a locally Gorenstein ring. Moreover, we 
characterize strict Gorenstein rings by the Cayley-Bacharach property
and the symmetry of their affine Hilbert function, as well as by the strict
Cayley-Bacharach property and the last difference of their affine
Hilbert function.
\end{abstract}


\maketitle

%
%

\section{Introduction}

\begin{flushright}
{\it History will, of course, go on repeating itself,}\\
{\it and the historians repeating each other.}
\end{flushright}

The Cayley-Bacharach Property (CBP) has a long and rich 
history. Classically, it has been formulated geometrically 
as follows: {\it A set of points~$\mathbb{X}$ in 
$n$-dimensional affine or projective space is said to have 
the Cayley-Bacharach
property of degree~$d$ if any hypersurface of degree~$d$
which contains all points of~$\mathbb{X}$ but one
automatically contains the last point.}
When~$d$ is one less than the regularity
index of the coordinate ring of~$\mathbb{X}$, we simply say 
that~$\mathbb{X}$ has the Cayley-Bacharach property (CBP).
Through the ages, the CBP has been shown for various, increasingly
general cases.
\begin{enumerate}
\item[(ca.\,320)\!] The classical theorem of Pappos (Pappus Alexandrinus) 
can be interpreted as a consequence of the fact that a set of 9 points in the plane 
which is the complete intersection of two curves consisting of three lines each,
has the~CBP (see~\cite{Pap}, Book 7, Prop.\ 139).

\item[(1640)] Pascal's theorem may be seen as a corollary of the fact that a set 
of 9 points in the plane, formed by intersecting a conic and a line 
with a set of three lines, has the~CBP (see~\cite{Pas}).

\item[(1748)] After being questioned by G.~Cramer about an apparent
{\it paradox} in the theory of plane curves, L.~Euler explained 
in~\cite{Eul} a solution which may be interpreted as claiming
that a complete intersection of two plane cubic curves consisting of 9 points
has the~CBP. 

\item[(1835)] From remarks of C.G.~Jacobi in~1835 (cf.~\cite{Jac2}, p.~331) 
and M.~Chasles in 1837 (cf.~\cite{Cha}, p.~150), it is clear that
by that time it was considered ``generally known'' that 9 points
in the plane which are the complete intersection of two curves of degree~3
have the~CBP.

\item[(1836)] In fact, based on his famous formula from~\cite{Jac1},
C.G.~Jacobi proved in~\cite{Jac2} an algebraic version of the~CBP
for a set of~$mn$ points in the plane which is a complete intersection
of a curve of degree~$m$ and a curve of degree~$n$.

\item[(1843)] In his paper~\cite{Cay1}, A.~Cayley stated a much stronger
property than the~CBP for such complete intersections which is, unfortunately,
incorrect in general. In fact, even for proving the~CBP, his argument contains 
a gap.

\item[(1885)] The first explicit statement and a correct proof of the~CBP
for reduced complete intersections in the plane were given by I.~Bacharach in~\cite{Bac}.
The proof was based on M.~Noethers ``Fundamentalsatz'' which is also known as
his $A\Phi+B\Psi$ Theorem. And even though A.\ Cayley failed to grasp the error
in his proof (see~\cite{Cay2}), the name ``Cayley-Bacharach Theorem''
became the commonly accepted one.

\item[(1952)] In the 1950s, starting with the work of D.~Gorenstein (see~\cite{Gor}),
it became clear that the~CBP is not restricted to
complete intersections, but it holds, for instance, for a set of points
in a projective space whose homogeneous coordinate ring is a Gorenstein ring.

\item[(1985)] A further significant step was taken in~\cite{DGO}
by E.~Davis, A.V.~Geramita and F.~ Orecchia, 
where the~CBP is extended to sets of points in~$\mathbb{P}^n$
whose coordinate rings are level algebras and
where arithmetically Gorenstein sets of points are characterized by the~CBP
and the symmetry of their Hilbert functions.

\item[(1993)] Some years later, in~\cite{GKR}, A.V.~Geramita together with the first
and third authors of this paper, showed that the~CBP of a set of points in~$\mathbb{P}^n$
is tied intrinsically to the structure of the canonical module of its homogeneous
coordinate ring.

\item[(1992)] The results of~\cite{DGO} and~\cite{GKR} were generalized by the
first author to arbitrary 0-dimensional subschemes of projective spaces over
an algebraically closed field (see~\cite{Kre1} and~\cite{Kre2}).

\item[(1996)] In~\cite{EGH}, D.\ Eisenbud, M.\ Green and J.\ Harris reviewed the
history of the Cayley-Bacharach theorem, put it in a general algebraic frame,
and proposed striking (and hitherto unproven) conjectures of vast
extensions. 

\item[(2016)] Thus it became clear that, in order to study even more general
versions of the~CBP, it is preferable to formulate it as a property
of the respective coordinate rings rather than sets of points or 0-dimensional
schemes. In this vein, the first and third authors defined in~\cite{KR3} the~CBP
for 0-dimensional affine algebras with a fixed presentation which have linear
maximal ideals, and they provided several algorithms to check it.

\item[(2015)] The most general definition of the~CBP to date
was given by the second author in~\cite{Lon} where he considered it
for presentations of arbitrary 0-dimensional affine algebras over 
arbitrary base fields.
\end{enumerate}

\smallskip
The definition in~\cite{Lon} is the starting point of this paper.
Our goal is to study this very general version of the~CBP and to
find efficient algorithms for checking it. A special emphasis will be
given to algorithms which will allow us to apply them to families
of 0-dimensional ideals parametrized by border basis schemes in a
follow-up paper (see~\cite{KLR}). Moreover, we generalize the main results about
the~CBP in~\cite{DGO}, \cite{GKR}, \cite{Kre1} and~\cite{Kre2}
to the most general setting of a 0-dimensional affine algebra over
an arbitrary base field.

To achieve these goals, we proceed as follows. Our main object of study
is a 0-dimensional affine algebra $R=P/I$ over an arbitrary field~$K$,
where we let $P=K[x_1,\dots,x_n]$ be a polynomial ring over~$K$ and~$I$ a
0-dimensional ideal in~$P$. Even if we do not specify it explicitly 
everywhere, we always consider~$R$ together with this fixed presentation.
In other words, we consider a fixed 0-dimensional subscheme 
$\mathbb{X}=\Spec(P/I)$ of~$\mathbb{A}^n$.

This corresponds to the classical setup. However, in the last decades
it has been customary to consider 0-dimensional subschemes of projective
spaces. Of course, via the standard embedding $\mathbb{A}^n \cong
D_+(x_0) \subset \mathbb{P}^n$, the classical setup can be
translated to this setting in a straightforward way. For instance, 
in this case the
affine coordinate ring~$R=K[x_1,\dots,x_n]/I$ has to substituted by 
the homogeneous coordinate ring $R^\hom =K[x_0,\dots,x_n]/I^\hom$, etc.
In this paper we use the affine setting for several reasons:
firstly, the ideals defining subschemes of~$\mathbb{X}$ can be
studied using the decomposition into local rings, 
secondly, the structure of the coordinate ring of~$\mathbb{X}$ and its canonical
module can be described via multiplication matrices,
and thirdly, the affine setup is suitable for generalizing
everything to families of 0-dimensional ideals via the
border basis scheme as in the upcoming paper~\cite{KLR}.

In Section~\ref{sec2} we start by recalling some basic properties
of~$I$ and~$R=P/I$. In particular, we recall the primary decomposition
$I=\Q_1\cap\cdots\cap \Q_s$ of~$I$, the corresponding primary
decomposition $\langle 0\rangle = \q_1\cap\cdots\cap \q_s$ of the
zero ideal of~$R$, and the decomposition $R=R/\q_1\times \cdots\times
R/\q_s$ of~$R$ into local rings. Then, for $i\in\{1,\dots,s\}$, a
minimal $\Q_i$-divisor $J$ of~$I$ is defined in such a way that
the corresponding subscheme of~$\mathbb{X}$ differs from~$\mathbb{X}$
only at the point $p_i=\mathcal{Z}(\M_i)$ and has the minimal possible
colength $\ell_i=\dim_K(P/\M_i)$, where $\M_i=\Rad(\Q_i)$.
For sets of points, these subschemes are precisely the sets $\mathbb{X}
\setminus \{p_i\}$ appearing in the classical formulation of the
Cayley-Bacharach Theorem.

Moreover, in order to have a suitable version of degrees, we recall
the degree filtration of~$R$, its affine Hilbert function $\HF^a_R$,
and its regularity index $\ri(R)$. As explained for instance in~\cite{KR2},
Section~5.6, the affine Hilbert function plays the role of the
usual Hilbert function if we consider affine algebras such as~$R$.

These constructions are combined in Section~\ref{sec3}.
We recall the definition and some characterizations
of separators from~\cite{KR3}. Then we show that a separator 
for a maximal ideal~$\m_i$ of~$R$ corresponds to a generator
of a minimal $\Q_i$-divisor~$J$ of~$I$, and we use the maximal
order of such a separator to describe the regularity index of~$J/I$.
Then the minimum of all regularity indices $\ri(J/I)$
is called the separator degree of~$\m_i$. 
We go on to show that this ``minimum of all maxima'' definition 
is the correct, but rather subtle generalization of the classical 
notion of the least degree
of a hypersurface containing all points of~$\mathbb{X}$ but~$p_i$.

The separator degree of a maximal ideal~$\m_i$ of~$R$ is bounded by
the regularity index $\ri(R)$, since the order of any separator
is bounded by this number. If all separator degrees attain
this maximum value, we say that~$R$ has the Cayley-Bacharach property (CBP),
or that $\mathbb{X}$ is a Cayley-Bacharach scheme.
In the last part of Section~3 we construct our first new algorithm
which allows us to check whether a given maximal ideal~$\m_i$
of~$R$ has maximal separator degree (see Proposition~\ref{CharCBgeneral} and 
Algorithm~\ref{alg:CheckCB}).

Although this algorithm can be used to check the~CBP of~$R$,
we construct a better one in Section~\ref{sec4}. 
It is based on the canonical module $\omega_R=\Hom_K(R,K)$
of~$R$. The module structure of~$\omega_R$ is given by
$(f\,\phi)(g)=\phi(fg)$ for all $f,g\in R$ and all $\phi\in\omega_R$.
It carries a degree filtration $\mathcal{G}=(G_i\omega_R)_{i\in\mathbb{Z}}$
which is given by $G_i\omega_R=\{ \phi\in\omega_R \mid \phi(F_{-i-1}R)=0\}$
and its affine Hilbert function which satisfies
$\HF^a_{\omega_R}(i)=\dim_K(R) - \HF^a_R(-i-1)$ for $i\in \mathbb{Z}$.
Generalizing some results in~\cite{GKR} and~\cite{Kre2}, we show
that the module structure of~$\omega_R$ is connected to the~CBP
of~$R$. More precisely, Theorem~\ref{CanModCB} says that~$R$
has the CBP if and only if $\Ann_R(G_{-\ri(R)}\omega_R)=\{0\}$.
Based on this characterization and the description of the
structure of~$R$ and the module structure of~$\omega_R$ via
multiplication matrices, we obtain the second main algorithm
of this paper, namely Algorithm~\ref{alg:CheckCBcanmod}
for checking the CBP of~$R$ using the canonical module.
As a nice and useful by-product, we show in Corollary~\ref{CBindependentonK}
that, for an extension field~$L$ of~$K$, the ring~$R$ 
has the~CBP if and only if $R \otimes_K L$ has the~CBP.

In Section~\ref{sec5} we turn our attention to 0-dimensional affine
algebras~$R$ which are locally Gorenstein and have the~CBP.
Extending some results in~\cite{KR3}, we show that~$R$
is locally Gorenstein if and only if~$\omega_R$ contains an element~$\phi$
such that $\Ann_R(\phi)=\{0\}$ and that we can check this effectively
(see Algorithm~\ref{alg:gor}). Then, in Theorem~\ref{CharGorCB},
we characterize locally Gorenstein rings having the~CBP
by the existence of an element $\phi\in\omega_{R\otimes L}$
of order $-\ri(R)$ such that $\Ann_{R\otimes L}(\phi)=\{0\}$.
Here we may have to use a base field extension $K\subseteq L$
or assume that~$K$ is infinite. This characterization
implies useful inequalities for the affine Hilbert function
of~$R$ (see Corollary~\ref{HFinequal}) and allows us to formulate and
prove Algorithm~\ref{alg:GorCB} which checks whether~$R$ is a locally
Gorenstein ring having the CBP using the multiplication matrices of~$R$.
To end this section, we characterize the CBP of~$R$ in the case when the last difference
$\Delta_R=\HF_R(\ri(R))-\HF_R(\ri(R)-1)$ is one (see Corollary~\ref{deltaqual1}).

The topic of the last section is to characterize 0-dimensional
affine algebras which are strict Gorenstein rings. This property means that
the graded ring $\grFR$ with respect to the degree filtration is 
a Gorenstein ring. In the projective case, the corresponding 0-dimensional
schemes are commonly called arithmetically Gorenstein.
Our first characterization of strict Gorenstein rings improves the
results in~\cite{DGO} and~\cite{Kre1}. More precisely, in Theorem~\ref{CharSGor1}
we show that~$R$ is strictly Gorenstein if and only if it has the CBP and
a symmetric affine Hilbert function. In particular, it follows that these rings
are locally Gorenstein. Then we define the strict CBP of~$R$ by
the CBP of $\grFR$ and show that it implies the CBP of~$R$ 
(see Proposition~\ref{SCBPimpliesCBP}).
Finally, we obtain a second characterization of strict Gorenstein rings:
in Theorem~\ref{CharSGor2} we prove that~$R$ is a strict Gorenstein ring if and
only if~$R$ has the strict CBP and $\Delta_R=1$.
Since strict complete intersections are strict Gorenstein rings,
this brings us full circle back to the historic origins of the CBP,
with the difference that now we can treat possibly non-reduced 
affine algebras with possibly non-rational support over arbitrary base fields.

All theorems and algorithms in this paper are amply illustrated by
non-trivial examples. These examples were calculated using the 
computer algebra system \cocoa~\cite{CoCoA}. Unless explicitly stated otherwise,
we use the definitions and notations given in~\cite{KR1}, \cite{KR2}, 
and~\cite{KR3}.

\bigbreak
%
%

\section{Zero-Dimensional Affine Algebras}\label{sec2}

Throughout this paper we let $K$ be a field and~$R$ a 0-dimensional
affine $K$-al\-ge\-bra. This means that~$R$ is a ring of the form
$R=P/I$, where $P=K[x_1,\dots,x_n]$ is a polynomial ring over~$K$
and~$I$ is a 0-dimensional ideal in~$P$. It is well-known that
in this case $R$, viewed as a $K$-vector space,
has finite dimension (see for instance~\cite{KR1}, Proposition 3.7.1).
Equivalently, we can take the geometric point of view and
consider the 0-dimensional subscheme $\mathbb{X}=\Spec(P/I)$
of the affine space $\mathbb{A}^n_K$ defined by~$I$. Then~$R$
is the affine coordinate ring of~$\mathbb{X}$.

Let us start by recalling some insights into the ring structure of~$R$
from~\cite{KR3}, Chapter~4, and fix the corresponding notation.

\begin{notation}
The ideal has a primary decomposition of the form
$$
I \;=\; \Q_1 \cap \cdots \cap \Q_s
$$
where the ideals $\Q_i$ are 0-dimensional primary ideals of~$P$
and are called the {\bf primary components} of~$I$.
The corresponding primes $\M_i=\Rad(\Q_i)$ are maximal ideals of~$P$.
They are called the {\bf maximal components} of~$I$.

The image of~$\Q_i$ in~$R$ will be denoted by~$\q_i$, and for the image
of~$\M_i$ in~$R$ we write~$\m_i$. Then the primary decomposition of the
zero ideal of~$R$ is given by $\langle 0\rangle = \q_1\cap\cdots\cap\q_s$,
and we have $\m_i=\Rad(\q_i)$ for $i=1,\dots,s$.

By applying the Chinese Remainder Theorem to this primary decomposition,
we obtain an isomorphism
$$
\imath:\;\; R \;\cong\; R/\q_1 \times \cdots\times R/\q_s
$$
which is called the {\bf decomposition of~$R$ into local rings}.
For $i=1,\dots,s$, the
ring $R_i=R/\q_i$ is a 0-dimensional local ring with maximal ideal
$\bar\m_i = \m_i/\q_i$. The ideal $\Soc(R_i)=\Ann_{R_i}(\bar\m_i)$
is called the {\bf socle} of~$R_i$. 

The field $L_i=R_i/\bar\m_i\cong R/\m_i$
is the {\bf residue field} of~$R_i$ and its $K$-vector space dimension
will be denoted by $\ell_i = \dim_K(L_i)$.
\end{notation}

The following proposition characterizes the smallest possible
non-zero ideals in~$R$, or equivalently, the smallest ideals in~$P$
strictly containing~$I$.

\begin{proposition}\label{charsmallid}
In the above setting let~$J$ be an ideal in~$P$ which contains~$I$
properly, and let $\bar J$ be its image in~$R$.
\begin{enumerate}
\item[(a)] The primary decomposition of the ideal~$\bar J$ is of the
form $\bar J = \q'_1\cap\cdots\cap \q'_s$, where 
$\q'_1,\dots,\q'_s$ are ideals in~$R$ such that $\q_j\subseteq \q'_j$
for $j=1,\dots,n$ and $\q_i\subset \q'_i$ for some $i\in\{1,\dots,n\}$.

\item[(b)] The primary decomposition of the ideal~$J$ is of the
form $\Q'_1\cap \cdots \cap \Q'_s$ where 
$\Q'_1,\dots,\Q'_s$ are ideals in~$P$ such that $\Q_j\subseteq \Q'_j$
for $j=1,\dots,n$ and $\Q_i\subset \Q'_i$ for some $i\in\{1,\dots,n\}$.

\item[(c)] For some $i\in \{1,\dots,s\}$, we have $\dim_K(\Q'_i/\Q_i) =
\dim_K(\q'_i/\q_i) \ge \ell_i$.

\item[(d)] If we have $\dim_K(\q'_i/\q_i)=\ell_i$ for some
$i\in\{1,\dots,s\}$ then every element $f\in \q'_i \setminus \q_i$
satisfies $\Ann_{R_i}(\bar f)=\bar\m_i$.
\end{enumerate}
\end{proposition}

\begin{proof}
To prove~(a), we apply the decomposition of~$R$ into local
rings. Then the ideal $\imath(\bar{J})$ is of the
form $\imath(\bar{J}) = J_1\times \cdots \times J_s$. This implies that
the ideals $\q'_i=\imath^{-1}(\langle 0\rangle \times \cdots \times
\langle 0\rangle \times J_i \times \langle 0\rangle \times \cdots \times
\langle 0\rangle)$ satisfy the claim.

Since claim~(b) follows immediately from~(a), we prove~(c) next.
For an element $f\in \q'_i\setminus \q_i$, we have
$\dim_K(\q'_i/\q_i)\ge \dim_K(\langle f\rangle/\q_i) = \dim_K(\langle
\bar f\rangle)$. Since $\bar f$ is a non-zero element of the local ring~$R_i$,
we get the inclusion $\Ann_{R_i}(\bar f)\subseteq \bar \m_i$.
This yields $\dim_K (\langle\bar f\rangle) =
\dim_K (R_i/\Ann_{R_i}(\bar f)) \ge \dim_K(R_i/\bar \m_i) = \ell_i$,
and the claim follows.

Finally we show~(d). Here all inequalities in the proof of~(c)
have to be equalities, and thus $\dim_K (R_i/\Ann_{R_i}(\bar f)) =
\dim_K(R_i/\bar \m_i)$ holds. Hence the
containment $\Ann_{R_i}(\bar f)\subseteq \bar \m_i$ is an equality, too.
\end{proof}

This proposition motivates the following definition.

\begin{definition}\label{MinimalQiDivisor}
In the above setting, let~$\ell_i = \dim_K(L_i)= \dim_K(P/\M_i)$
for $i=1,\dots,s$.
\begin{enumerate}
\item[(a)]  An ideal $J$ in $P$ is called a {\bf $\Q_i$-divisor} of~$I$ if~$J$
is of the form $J = \Q_1\cap\cdots\cap\Q_i'\cap\cdots\cap \Q_s$
with an ideal $\Q'_i$ in~$P$ such that $\Q_i\subset \Q_i'\subseteq \M_i$.

\item[(b)] An ideal $J$ in $P$ is called a {\bf minimal $\Q_i$-divisor} of~$I$
if it is a $\Q_i$-divisor of~$I$ and $\dim_K(J/I) = \ell_i$.
\end{enumerate}
\end{definition}

Using the decomposition of $R$ into local rings, we deduce that
if the ideal~$J$ is a $\Q_i$-divisor of~$I$ then 
$\ell_i=\dim_K(J/I) = \dim_K(\Q_i'/\Q_i)$.
Let us also translate this definition into the language of Algebraic
Geometry (see~\cite{Lon} and \cite{KL}).

\begin{definition}
Let $\mathbb{X}$ be the 0-dimensional subscheme of~$\mathbb{A}^n_K$
defined by~$I$.
\begin{enumerate}
\item[(a)] The set $\Supp(\mathbb{X}) = \{p_1,\dots,p_s\}$,
where $p_i=\mathcal{Z}(\M_i)$ for $i=1,\dots,s$, is called
the {\bf support} of~$\mathbb{X}$.

\item[(b)] Given $i\in\{1,\dots,s\}$, a subscheme~$\mathbb{Y}$
of~$\mathbb{X}$ is called a {\bf $p_i$-subscheme} of~$\mathbb{X}$ if
$\mathcal{O}_{\mathbb{Y},p_j} = \mathcal{O}_{\mathbb{X},p_j}$
for every $j\ne i$.

\item[(c)] Given $i\in\{1,\dots,s\}$, a $p_i$-subscheme~$\mathbb{Y}$
of the scheme~$\mathbb{X}$ is said to be a {\bf maximal $p_i$-subscheme} if
$\deg(\mathbb{Y})=\deg(\mathbb{X})-\ell_i$.
\end{enumerate}
\end{definition}

Clearly, the defining ideal of a $p_i$-subscheme $\mathbb{Y}$
of~$\mathbb{X}$ is a $\Q_i$-divisor of~$I$, and vice versa.
Moreover, since $\deg(\mathbb{X})=\dim_K(R)$, a maximal
$p_i$-subscheme of~$\mathbb{X}$ corresponds to a minimal
$\Q_i$-divisor. Let us see an example.

\begin{example}\label{exQideal}
Let $K$ be a field, let $P = K[x,y]$, and let $\Q = \langle x^2, y^2\rangle$.
Clearly, the ideal $\Q$ is $\M$-primary for $\M=\langle x,y\rangle$,
and we have $\ell=\dim_K(P/\M)=1$.

Now we consider the ideal $J_1= \Q + \langle x\rangle =
\langle x,y^2\rangle$. Clearly $J_1$ is $\M$-primary and hence
a $\Q$-divisor of~$\Q$. Since we have $\dim_K(J_1/\Q)=2>\ell$, the
ideal~$J_1$ is not a minimal $\Q$-divisor of~$\Q$.

Next we look at the ideal $J_2=\Q+\langle xy\rangle = \langle x^2,xy,y^2\rangle$.
Again it is clear that $J_2$ is $\M$-primary, and therefore a
$\Q$-divisor of~$\Q$. In this case we get the equality $\dim_K(J_2/\Q)=1=\ell$,
whence $J_2$ is even a minimal $\Q$-divisor of~$\Q$.
\end{example}

Useful invariants of a 0-dimensional affine algebra
are given by the values of its affine Hilbert function which
we recall next. For this purpose, we equip~$P$ with the
(standard) {\bf degree filtration} $\widetilde{\mathcal{F}}
= (F_iP)_{i\in\mathbb{Z}}$, where
$$
F_iP \;=\;  \{f\in P \mid \deg(f) \le i\} \;\cup\; \{0\}
$$
This is an increasing filtration which satisfies
$F_iP=\{0\}$ for $i<0$ and $F_0P=K$.
For every $i\in\mathbb{Z}$, let $F_iI=F_iP\cap I$, and let
$F_iR = F_iP/F_iI$. Then the family $(F_iI)_{i\in\mathbb{Z}}$
is called the {\bf induced filtration} on~$I$, and the family
$\mathcal{F}= (F_iR)_{i\in\mathbb{Z}}$
is a $\mathbb{Z}$-filtration on~$R$
which is called the {\bf degree filtration} on~$R$.
Note that we have $\bigcup_{i\in\mathbb{Z}} F_iP = P$ and
$\bigcup_{i\in\mathbb{Z}} F_iI = I$, and hence
$\bigcup_{i\in\mathbb{Z}} F_iR = R$.
Since $R$ is 0-dimensional, the degree filtration on~$R$
has only finitely many distinct parts, and we have $F_iR =R$
for $i\gg 0$.

\begin{definition}
Let $R=P/I$ be a 0-dimensional affine $K$-algebra as above.
\begin{enumerate}
\item[(a)] The {\bf affine Hilbert function} of~$R$ is
defined as the map
$$
\HF^a_R: \mathbb{Z}\longrightarrow \mathbb{Z} \hbox{\rm\quad given by\quad}
i\longmapsto \dim_K(F_iR)
$$

\item[(b)] The number $\ri(R)=\min\{i\in\mathbb{Z} \mid
\HF^a_R(j)=\dim_K(R)\hbox{\ \rm  for all\ }j\ge i\}$
is called the {\bf regularity index} of~$R$.

\item[(c)]
The first difference function $\Delta \HF^a_R(i)= \HF^a_R(i)-
\HF^a_R(i-1)$ of the affine Hilbert function of~$R$
is called the {\bf Castelnuovo function} of~$R$.

\item[(d)] The number $\Delta_R=\Delta \HF^a_R(\ri(R))$
is called the {\bf last difference} of~$\HF^a_R$ (or of~$R$).
\end{enumerate}
\end{definition}

It is easy to see that we have $\HF^a_R(i)=0$ for $i<0$ and a chain of
inequalities
$$
1 = \HF^a_R(0) < \HF^a_R(1) < \cdots < \HF^a_R(\ri(R)) = \dim_K(R)
$$
So, the degree filtration on~$R$ is increasing, exhaustive, and
{\bf orderly} in the sense that every non-zero element has an order
according to the following definition (see also~\cite{KR2}, 6.5.10).

\begin{definition}
For $f\in R\setminus \{0\}$, let
$\ord_\Fbar(f)=\min \{i\in\mathbb{Z} \mid f\in F_i R \setminus
F_{i-1}R\}$. This number is called the {\bf order} of~$f$
with respect to~$\Fbar$.
\end{definition}

From the above description of~$\HF^a_R$ it follows that we have
$0\le \ord_\Fbar(f)\le \ri(R)$ for every $f\in R\setminus \{0\}$.
The order of an element represents the smallest degree of
a representative of its residue class modulo~$I$.

\begin{remark}\label{computeOrd}
This description points us to an easy way to calculate
the order of an element: if $F\in P$ represents an element
$f\in R\setminus \{0\}$ and $\sigma$ is a degree compatible term
ordering, then $\ord_{\Fbar}(f)$ is given by the degree
of the normal form $\NF_{\sigma,I}(F)$ (see~\cite{KR1}, Def.~2.4.8).
\end{remark}

For actual computations involving the affine Hilbert function of~$R$,
we like to have the following kind of $K$-basis.

\begin{definition}\label{defdegfiltered}
Let $d=\dim_K(R)$, and let $B=(b_1,\dots,b_d)\in R^d$.
The tuple~$B$ is called a {\bf degree filtered $K$-basis} of~$R$
if $F_iB = B \cap F_i R$ is a $K$-basis of~$F_iR$ for every $i\in\mathbb{Z}$,
and if we have ${\ord_\Fbar(b_1) \le \ord_\Fbar(b_2) \le \cdots \le \ord_\Fbar(b_d)}$.
\end{definition}

In the sequel we assume that $b_1=1$ in each degree filtered basis
$B=(b_1,\dots,b_d)$.

\begin{remark}\label{ordinvariant}
For every degree filtered basis $B=(b_1,\dots,b_d)$
and for every $i\in\{0,\dots,\ri(R)\}$, we have 
$$
\HF^a_R(i) \;=\; \# \{ j\in\{1,\dots,d\} \mid
\ord_\Fbar(b_j)\le i \}
$$
In particular,
the tuple $(\ord_\Fbar(b_1) ,\dots, \ord_\Fbar(b_d))$
is independent of the choice of the degree filtered basis~$B$.
\end{remark}

Remark~\ref{computeOrd} and
Remark~4.6.4 in~\cite{KR3} provide several ways of computing
a degree-filtered basis of~$R$. Moreover,
given a degree filtered basis $B=(b_1,\dots,b_d)$ and an
element $g\in R\setminus \{0\}$, we write $g=a_1 b_1 + \cdots +
a_d b_d$ with $a_i\in K$ for $i=1, \dots, d$ and have the equality
$\ord_\Fbar(g)= \max\{ \ord_\Fbar(b_i) \mid i \in \{1,\dots,d\}
\hbox{\ \rm and\ } a_i\ne 0\}$.

\smallskip
The following example provides a monomial $K$-basis 
which is not degree-filtered.

\begin{example}\label{nondegreefiltered}
Let $K=\mathbb{Q}$, let $P = K[x,y]$,
let $I$ be the vanishing ideal of the affine set
of eight points given by
$p_1=(1,-1)$, $p_2 =(0,2)$, $p_3=(1,1)$,
$p_4 =(1,2)$, $p_5=(0,1)$, $p_6=(1,3)$,
$p_7 = (2,4)$, and $p_8 =(3,4)$, and let $R=P/I$.
The reduced Gr\"{o}bner basis of~$I$ with respect to
\texttt{DegRevLex} is
\begin{gather*}
\{ \; x^{2}y -4x^{2} - xy + 4x,\quad 
x^{3} + xy^{2} -6x^{2} -3xy - y^{2} + 7x + 3y -2,\\
y^{4} -10xy^{2} -5y^{3} + 15x^{2} + 30xy + 15y^{2} -35x -25y + 14,\\
xy^{3} -7xy^{2} - y^{3} + 14xy + 7y^{2} -8x -14y + 8\;\}
\end{gather*}
Since this term ordering is degree compatible,
the residue classes of the elements in the tuple
$(1, y, x, y^2, xy, x^2, y^3, xy^2)$
form a degree-filtered $K$-basis of~$R$
with order tuple $(0,1,1,2,2,2,3,3)$.
On the other hand, the reduced Gr\"{o}bner basis of~$I$ 
with respect to~\texttt{Lex} is
\begin{gather*}
\{\, x^{2} -\tfrac{2}{3}xy^{2} + 2xy -\tfrac{7}{3}x +
\tfrac{1}{15}y^{4} -\tfrac{1}{3}y^{3}
+ y^{2} -\tfrac{5}{3}y + \tfrac{14}{15},\\
xy^{3} -7xy^{2} + 14xy -8x - y^{3} + 7y^{2} -14y + 8,\,
y^{5} -9y^{4} + 25y^{3} -15y^{2} -26y + 24 \,\}
\end{gather*}
So, the residue classes of the elements in the tuple
$B=(1, y, x,  y^2, xy, y^3, xy^2, y^4)$ form a $K$-basis of~$R$.
Since $\bar{y}^4 = 10\bar{x}\bar{y}^2
+ 5\bar{y}^3 - 15\bar{x}^2 - 30\bar{x}\bar{y}
- 15\bar{y}^2 + 35\bar{x} + 25\bar{y} - 14$, 
we have $\ord_\Fbar(\bar{y}^4)=3$.
Altogether, we see that~$B$ is not a degree-filtered basis, since
its order tuple is $(0,1,1,2,2,3,3,3)$.
\end{example}

\bigbreak
%
%

\section{Separators and the Cayley-Bacharach Property}\label{sec3}

In this section we continue to use the notation introduced above.
In particular, we let $R=P/I$ be a 0-dimensional affine $K$-algebra
whose zero ideal has the primary decomposition $\langle 0\rangle =
\q_1\cap \cdots \cap \q_s$, and we let $\m_i=\Rad(\q_i)$
be the maximal ideals of~$R$ for $i=1,\dots,s$.
The following definition generalizes the ones in~\cite{GKR}
and~\cite{Kre2}.

\begin{definition}\label{DefSeparator}
For $i\in\{1,\dots,s\}$, an element $f\in R$ is called a
{\bf separator} for~$\m_i$ if we have
$\dim_K\langle f \rangle= \dim_K(R/\m_i)$ and
$f \in \q_j$ for every $j\ne i$.
\end{definition}

The following characterizations of separators were shown
in~\cite{KR3}, Theorem 4.2.11.

\begin{theorem}{\bf (Characterization of Separators)}\label{charsep}\\
Let $R$ be a 0-dimensional affine $K$-algebra and let $i\in\{1,\dots,s\}$.
For an element $f\in R$, the following conditions are equivalent.
\begin{enumerate}
\item[(a)] The element~$f$ is a separator for~$\m_i$.

\item[(b)] We have $\Ann_R(f)=\m_i$.

\item[(c)] The element~$f$ is a non-zero element of
$(\q_i :_{{}_R} \m_i)\cdot \prod_{j\ne i} \q_j$.

\item[(d)] The image of~$f$ is a non-zero element in the socle of the
local ring $R/\q_i$ and, for $j\ne i$, the image of~$f$
is zero in~$R/\q_j$.
\end{enumerate}
\end{theorem}

Using the language of Definition~\ref{MinimalQiDivisor},
we can rephrase Definition~\ref{DefSeparator} as follows.

\begin{proposition}\label{separatoranddivisor}
Let $F\in P$, let $f=F+I$ be the residue class of~$F$ in~$R$,
and let $i\in\{1,\dots,s\}$. Then the following conditions
are equivalent.
\begin{enumerate}
\item[(a)] The element $f$ is a separator for~$\m_i$.

\item[(b)] The ideal $J=I+\langle F\rangle$ is a
minimal $\Q_i$-divisor of~$I$.
\end{enumerate}
\end{proposition}

\begin{proof}
To prove that~(a) implies~(b), we note that Theorem~\ref{charsep}.c
implies that $J=I+\langle F\rangle$ is a $\Q_i$-divisor.
Since we have $\dim_K(J/I)=\dim_K\langle f\rangle=\ell_i$,
it is a minimal $\Q_i$-divisor of~$I$. Conversely, we have
$J=I+\langle F\rangle =
\Q_1 \cap \cdots \cap \Q'_i \cap\cdots \cap \Q_s$,
and therefore $f\in \q_j$ for $j\ne i$. Moreover, the condition
$\ell_i =\dim_K(\Q'_i/\Q_i)=\dim_K(\q'_i/\q_i)=\dim_K(\langle \bar f\rangle)$
and Proposition~\ref{charsmallid}.d yield $\Ann_{R/\q_i}(\bar f)=\bar \m_i$,
i.e., the fact that~$\bar f$ is an element of the socle of~$R/\q_i$.
\end{proof}

\begin{remark}\label{twokinds}
Given a maximal ideal~$\m_i$ of~$R$, the separators for~$\m_i$
may not be uniquely determined in two different ways:
\begin{enumerate}
\item[(1)] It is possible that two separators $f,g$ for~$\m_i$
correspond to the same minimal $\Q_i$-divisor of~$I$.
In this case, the ideals $\langle \bar f\rangle$
and $\langle \bar g\rangle$ in~$R/\q_i$ are equal,
but if we have $\ell_i=\dim_K(R/\m_i)>1$,
the orders of~$f$ and~$g$ with respect to~$\Fbar$
may not be equal.

\item[(2)] If $\dim_K(\Soc(R/\q_i))>\ell_i$, there
exist separators $f,g$ for~$\m_i$ which
correspond to different $\Q_i$-divisors of~$I$.
In this case, the ideals $\langle \bar f\rangle$ and 
$\langle\bar g\rangle$ in~$R/\q_i$ are not equal. 
\end{enumerate}
\end{remark}

The following example demonstrates these two kinds of non-uniqueness.

\begin{example}\label{twoSources}
Let $K= \mathbb{Q}$, let $P = K[x,y]$, let~$I$ be the ideal of~$P$ 
generated by $\{ xy,\; y^3,\; x^4+x^2\}$, and let $R=P/I$.
The primary decomposition of~$I$ is given by $I = \Q_1\cap\Q_2$, 
where $\Q_1 = \langle y,\; x^2+1 \rangle$
and $\Q_2 = \langle xy,\; x^2,\: y^3\rangle$.
Clearly, the corresponding maximal components
are the ideals $\M_1 = \Rad(\Q_1) = \Q_1$ and $\M_2 = 
\Rad(\Q_2) = \langle x,\; y\rangle$.
The affine Hilbert function of~$R$ is $(1,3,5,6, 6,\dots)$,
and hence $\ri(R) = 3$.

The residue classes  of~$x^2$ and~$x^3$ in~$R$ are separators for~$\m_1$.
Their orders are~$2$ and~$3$, respectively, since they coincide with their 
normal forms with respect to any degree compatible term ordering.
Thus the equality 
$\Q_1+\langle x^2\rangle = \Q_1 + \langle x^3\rangle =\langle 1\rangle$
shows that this is a case of non-uniqueness of the first kind.

Since we have 
$(\Q_2:\M_2)\cap \Q_1 = \langle y^2,\; xy,\; x^3+x  \rangle$,
the residue classes of~$y^2$ and~$x^3+x$ in~$R$ are separators for~$\m_2$.
Their orders are~$2$ and~$3$, respectively, because they coincide with their 
normal forms with respect to any degree compatible term ordering.
Notice that the two ideals
$\Q_2+\langle y^2\rangle$ and  $\Q_2 + \langle x^3+x\rangle$ are different.
Consequently, this is a case of non-uniqueness of the second kind.
\end{example}

Keeping these sources of non-uniqueness in mind, 
we introduce the following notion.

\begin{definition}
Let $R=P/I$ be a 0-dimensional affine $K$-algebra as above,
let $\m_1,\dots,\m_s$ be the maximal ideals of~$R$, and let
$i\in\{1,\dots,s\}$.
Given a minimal $\Q_i$-divisor~$J$ of~$I$
and its image $\bar{J}$ in~$R$, we let
$$
\ri(\bar{J}) \;=\; \max\{ \ord_\Fbar(f) \mid f\in \bar{J}
\setminus \{0\} \}
$$
Then the number
$$
\sepdeg(\m_i)= \min\{ \ri(\bar{J}) \mid J\hbox{\rm\ is a minimal
$\Q_i$-divisor of\ }I\}
$$
is called the {\bf separator degree} of~$\m_i$ in~$R$.
\end{definition}

The following proposition shows that this definition is
justified in the sense that it agrees with previous
definitions in~\cite{KR2} and~\cite{KR3}.

\begin{proposition}
Let $R=P/I$ be a 0-dimensional affine $K$-algebra as above, 
and let $\m_1,\dots,\m_s$ be the maximal ideals of~$R$.
Moreover, let $i\in\{1,\dots,s\}$, let~$J$
be a minimal $\Q_i$-divisor of~$I$, let $\bar{J}$
be the image of~$J$ in~$R$, and let
$\HF^a_{\bar{J}}(j)=\dim_K(\bar{J} \cap F_jR)$ for $j\in\mathbb{Z}$
be the affine Hilbert function of~$\bar{J}$.
\begin{enumerate}
\item[(a)]
We have $\ri(\bar{J}) = \min\{j\in\mathbb{Z} \mid \HF^a_{\bar{J}}(j)=
\ell_i\}$ and $\HF^a_{\bar{J}}(j)=\ell_i$ 
for all $j\ge \ri(\bar{J})$. In other words, the number
$\ri(\bar{J})$ is the regularity index of the Hilbert function
of~$\bar{J}$ in the sense of~\cite{KR2}, Definition 5.1.8.

\item[(b)] If the maximal ideal~$\m_i$ is linear, then its separator
degree satisfies the equality $\sepdeg(\m_i)=\min \{ \ord_{\Fbar}(f) \mid
f\hbox{\rm \ is a separator for\ }\m_i\}$.
In other words, in this case the number $\sepdeg(\m_i)$ agrees with
the one defined in~\cite{KR3}, Definition 4.6.10.b.
\end{enumerate}
\end{proposition}

\begin{proof}
First we prove claim~(a). By their definition, affine Hilbert functions
are non-decreasing, and by Definition~\ref{MinimalQiDivisor},
we have $\HF^a_{\bar{J}}(j)=\ell_i$ for $j\gg 0$. If an element
$f\in\bar{J}$ satisfies $\ord_\Fbar(f)=j$ for some $j\in\mathbb{Z}$
then we have $f\notin \bar{J}\cap F_{j-1}R$, and therefore
$\HF^a_{\bar{J}}(j)> \HF^a_{\bar{J}}(j-1)$. This implies the claim.

Next we show~(b). Since $\m_i$ is linear, we have $\ell_i=1$, and
therefore $\bar{J}=K\cdot f$ for every $f\in\bar{J}\setminus \{0\}$.
Hence we get $\ri(\bar{J})=\ord_\Fbar(f)$, and the claim follows.
\end{proof}

For the separator degree of a maximal ideal of~$R$, we have
the following bound.

\begin{proposition}\label{inequalitiesforri}
Let $R=P/I$ be a 0-dimensional affine $K$-algebra whose maximal
ideals are $\m_1,\dots,\m_s$, and let $i\in\{1,\dots,s\}$.
\begin{enumerate}
\item[(a)] Given a minimal $\Q_i$-divisor~$J$ of the ideal~$I$ and its image~$\bar{J}$
in~$R$,  we have  the inequality $\ri(\bar{J}) \le \ri(R)$. In particular, we have
$\ord_\Fbar(f)\le \ri(R)$ for every non-zero element $f \in \bar{J}$.

\item[(b)] We have $\sepdeg(\m_i)\le \ri(R)$.

\end{enumerate}
\end{proposition}

\begin{proof}
Claim~(a) follows from the observation that the equality $F_{\ri(R)}R
=R$ implies the equality $J\cap F_{\ri(R)}R = J$. Claim~(b)
follows from~(a) and the definition of~$\sepdeg(\m_i)$.
\end{proof}

This proposition allows us to characterize maximal separator
degrees as follows.

\begin{corollary}\label{sepdegandord}
Under the assumptions of the proposition,
the following conditions are equivalent.
\begin{enumerate}
\item[(a)] For every minimal $\Q_i$-divisor~$J$ of~$I$ and its image 
$\bar{J}$ in~$R$, there is a generator~$f$ of~$\bar{J}$ 
such that $\ord_{\Fbar}(f) = \ri(R)$.

\item[(b)] For every minimal $\Q_i$-divisor~$J$ of~$I$ and its image 
$\bar{J}$ in~$R$, we have the equality $\ri(\bar{J}) = \ri(R)$.

\item[(c)] We have $\sepdeg(\m_i) = \ri(R)$.
\end{enumerate}
\end{corollary}

\begin{proof}
The proof follows easily from the definitions 
and part~(a) of the proposition.
\end{proof}

The preceding corollary suggests the following definition.

\begin{definition}
Let $R=P/I$ be a 0-dimensional affine $K$-algebra, and let
$\mathbb{X}=\Spec(P/I)$ be the 0-dimensional affine scheme
defined by~$I$. We say that~$R$ has the {\bf Cayley-Bacharach property (CBP)},
or that $\mathbb{X}$ is a {\bf Cayley-Bacharach scheme},
if the equivalent conditions of the above corollary are satisfied.
\end{definition}

By the above results, this definition agrees with the usual
definition given in~\cite{Kre2}, Section~2 and~\cite{KR3}, Definition
4.6.12 if all maximal ideals are linear. Let us also rephrase
it using Hilbert functions of subschemes of~$\mathbb{X}$.
Here the affine Hilbert function of a scheme is the affine Hilbert function
of its coordinate ring.

\begin{corollary}
Let $\mathbb{X}$ be a 0-dimensional subscheme of~$\mathbb{A}^n_K$
with affine coordinate ring $R=P/I$, let $\Supp(\mathbb{X})=
\{p_1,\dots,p_s\}$, and let $\ell_i=\dim_K(\mathcal{O}_{\mathbb{X},p_i}
/\m_{\mathbb{X},p_i})$ for $i=1,\dots,s$. Then the following conditions
are equivalent.
\begin{enumerate}
\item[(a)] The scheme~$\mathbb{X}$ is a Cayley-Bacharach scheme.

\item[(b)] For $i\in\{1,\dots,s\}$ and for every maximal
$p_i$-subscheme $\mathbb{Y}$ of~$\mathbb{X}$, we have
$\HF^a_{\mathbb{Y}}(\ri(R)-1) > \HF^a_{\mathbb{X}}(\ri(R)-1)
-\ell_i$.
\end{enumerate}
\end{corollary}

\begin{proof}
Let $J$ be the vanishing ideal of a maximal $p_i$-subscheme~$\mathbb{Y}$
of~$\mathbb{X}$, and let~$\bar{J}$ be its image in~$R$.
By definition, we have
$\HF^a_{\mathbb{X}}(j) - \HF^a_{\mathbb{Y}}(j) = \HF^a_{\bar{J}}(j)$
for all $j\in\mathbb{Z}$.
Hence the inequality in~(b) can be rewritten as $\HF^a_{\bar{J}}(\ri(R)-1) <\ell_i$.
Therefore the inequality in~(b) is equivalent to $\ri(\bar{J}) = \ri(R)$, and the 
conclusion follows from Corollary~\ref{sepdegandord}.
\end{proof}

In the last part of this section we discuss methods for
computing the separator degree and for
checking whether the separator degree of
a given maximal ideal of~$R$ is maximal.
For this purpose it is convenient to introduce the following terminology.

\begin{definition}
Let $R=P/I$ be a 0-dimensional affine $K$-algebra,
and let $\imath:\, R\longrightarrow R/\q_1 \times\cdots\times R/\q_s$
be its decomposition into local rings. For $i=1,\dots,s$, let
$S_i$ be the preimage
$$
S_i \;=\; \imath^{-1} \big( \{0\} \times \cdots\times \{0\} \times
\Soc(R/\q_i) \times \{0\} \times\cdots\times \{0\} \big)
$$
Then the $K$-vector spaces $S_i$ are called the {\bf socle spaces}
of~$R$.
\end{definition}

Notice that the socle spaces~$S_i$ are ideals in~$R$ having a very
particular structure. First of all, the space~$S_i$ is annihilated by the
maximal ideal~$\m_i$. Thus it is a finite dimensional vector space
over the field $L_i=P/\M_i$. Letting $k_i=\dim_{L_i}(S_i)$ and recalling
that $\ell_i=\dim_K(L_i)$, we have $\dim_K(S_i)=k_i\ell_i$
for $i=1,\dots,s$. Secondly, by Theorem~\ref{charsep},
the non-zero elements of~$S_i$ are precisely the separators of~$\m_i$.

In the following case, the separator degree of a maximal ideal of~$R$
is easy to calculate. Recall that a local ring $(A,\m)$ is called a 
{\bf Gorenstein ring} if we have $\dim_{A/\m}\Soc(A)=1$.

\begin{remark}\label{sepdegprop}
Assume that $R/\q_i$ is a Gorenstein ring for some $i\in\{1,\dots,s\}$.
Then every non-zero element $f\in S_i$ generates~$S_i$.
Let $\{e_1,\dots,e_\ell\}$ be a set of elements in~$R$ whose residue classes
in~$L_i=R/\m_i$ form a $K$-basis of~$L_i$.
Then the elements in $\{e_1f,\dots,e_\ell f\}$ form a $K$-basis
of the ideal $\langle f\rangle$ in~$S_i$.
Consequently, we have $\sepdeg(\m_i)= \max\{ \ord_\Fbar(e_i f) \mid
i=1,\dots,\ell\}$.
\end{remark}

If~$R_i$ is not a Gorenstein ring, the situation
is somewhat more complicated, since the separators of~$\m_i$ generate
many different ideals in~$R$. If~$K$ is infinite, there are even
infinitely many such ideals. Nevertheless, the following
proposition allows us to characterize when the separator degree
of a maximal ideal of~$R$ is maximal.
Recall that the {\bf leading form} of a non-zero
element~$f$ of~$R$ with $\ord_\Fbar(f)=\gamma$ is defined as the
residue class $\LF_\Fbar(f) = f+ F_{\gamma-1}R$ in
$F_\gamma R/F_{\gamma-1}R$ (see~\cite{KR3}, Definition 6.5.10).

Moreover, notice that in order to construct a $K$-basis of
a socle space~$S_i$ of~$R$, we may find an $L_i$-basis
$\{s_1,\dots,s_{k_i}\}$ of~$S_i$ and a $K$-basis 
$\{e_1,\dots,e_{\ell_i}\}$ of~$L_i$. Then the set of products
$\{e_\lambda s_\kappa \mid 1\le\lambda\le \ell_i,\;
1\le \kappa\le k_i\}$ is a $K$-basis of~$S_i$
which consists of $\ell_i k_i$ elements.

\begin{proposition}{\bf (Characterization of Maximal Separator
Degrees)}\label{CharCBgeneral}\\
Let $R=P/I$ be a 0-dimensional affine $K$-algebra,
let $\m_1,\dots,\m_s$ be the maximal ideals of~$R$, and let
$i\in\{1,\dots,s\}$. Let $k_i=\dim_{L_i}(S_i)$
and $m_i=\ell_i k_i$.
Let $(e_1,\dots,e_{\ell_i})$ be a $K$-basis of~$L_i$, let
$(f_1,\dots,f_{m_i})$ be a $K$-basis of~$S_i$, and let
$M_i = (L(e_j f_k))_{j,k}$ be the matrix in
$\Mat_{\ell_i, m_i}(F_{\ri(R)}R / F_{\ri(R)-1}R)$ such that
$$
L(e_j f_k) \;=\; \begin{cases}
\LF_\Fbar(e_j f_k) & \hbox{\it if\ }\ord_\Fbar(e_j f_k)=\ri(R),\\
0 & \hbox{\it otherwise}.
\end{cases}
$$
Then we have $\sepdeg(\m_i)=\ri(R)$ if and only
if the columns of~$M_i$ are $K$-linearly independent.
\end{proposition}

\begin{proof}
By definition, we have $\sepdeg(\m_i)<\ri(R)$ if and
only if there exists an ideal $\bar{J}$ in~$S_i$
such that $\dim_K(\bar{J})=\ell_i$ and
$\bar{J}\subseteq F_{\ri(R)-1}R$. Every ideal~$\bar{J}$
with $\dim_K(\bar{J})=\ell_i$ is generated by a separator
$g\in S_i$. 
Let us write $g=\sum_{k=1}^{m_i} a_{k}f_k$
with $a_{k}\in K$. The condition $\ord_\Fbar(g)<\ri(R)$
is equivalent to the condition that
$\sum_{k=1}^{m_i} a_k L(f_k) =0$ in the vector space
$F_{\ri(R)}R / F_{\ri(R)-1}R$.
Since the ideal $\bar{J}=\langle g\rangle$ is contained
in~$F_{\ri(R)-1}R$ if and only if
$\ord_\Fbar(e_j g) <\ri(R)$ for $j=1,\dots,\ell_i$, this
condition is equivalent to the condition that
$\sum_{k=1}^{m_i} a_{k} L(e_j f_k) = 0$ for $j=1,\dots,\ell_i$.
In other words, we have $\sepdeg(\m_i) <\ri(R)$ if and only
if there exists a tuple $(a_k)\in K^{m_i}\setminus  \{0\}$ 
such that $M_i\cdot (a_k)\tr =0$. This proves the claim.
\end{proof}

In view of this proposition, we have the following algorithm
for checking whether the separator degree of some maximal ideal
of~$R$ attains its maximal value $\ri(R)$.

\begin{algorithm}{\bf (Checking Maximal Separator
Degrees)}\label{alg:CheckCB}\\
Let $R=P/I$ be a 0-dimensional affine $K$-algebra,
let $\q_1,\dots,\q_s$ be the primary components of the
zero ideal in~$R$, let $\m_1,\dots,\m_s$ be the corresponding
maximal ideals of~$R$, and let $i\in\{1,\dots,s\}$.
Consider the following sequence of instructions.
\begin{enumerate}
\item[(1)] Compute a $K$-basis $(e_1,\dots,e_{\ell_i})$
of the field $L_i=R/\m_i$.

\item[(2)] Calculate the socle $\Soc(R/\q_i)$. Using the
Chinese Remainder Theorem, compute a $K$-basis 
$B_i = (f_1,\dots,f_{m_i})$ of the preimage
$$
S_i \;=\;\imath^{-1}( \{0\} \times\cdots\times \{0\}
\times \Soc(R/\q_i) \times \{0\} \times \cdots\times \{0\})
$$
under the isomorphism $\imath:\; R\cong R/\q_1 \times
\cdots\times R/\q_s$.

\item[(3)] Calculate a $K$-basis $V=(v_1,\dots,v_\Delta)$
of $F_{\ri(R)}R / F_{\ri(R)-1}R$.

\item[(4)] Form the matrix $M_i$ in
$\Mat_{\ell_i \Delta, m_i}(K)$ which is defined as follows:
For $j=1,\dots,\ell_i$, compute the column vectors in~$K^\Delta$
containing the coordinates of $L(e_jf_k)$ with respect to the
basis~$V$ and put them into the $j$-th block of rows of~$M_i$.
Here $L(e_jf_k)$ is defined as in Proposition~\ref{CharCBgeneral}.

\item[(5)] If the rank of~$M_i$ is~$m_i$,
return ${\tt TRUE}$. Otherwise, return ${\tt FALSE}$.
\end{enumerate}
This is an algorithm which checks whether the maximal ideal~$\m_i$
of~$R$ has maximal separator degree and returns the corresponding
Boolean value.
\end{algorithm}

\begin{proof}
The finiteness of this algorithm is clear. The correctness
follows from Remark~\ref{sepdegprop} and
Proposition~\ref{CharCBgeneral}.
\end{proof}

Let us apply this algorithm in a concrete case.

\begin{example}
Let $R=P/I$ be the 0-dimensional affine $K$-algebra given in
Example~\ref{twoSources}. The generating set $\{xy,y^3,x^4+x^2\}$
is also the reduced Gr\"{o}bner basis of $I$ with respect to~\texttt{DegRevLex}.
So, the residue classes of the elements in the tuple
$(1,y,x,y^2,x^2,x^3)$ form a degree filtered $K$-basis of $R$.
In particular, we have $\ri(R)=3$ and $\Delta_R=1$.
Now we want to apply Algorithm~\ref{alg:CheckCB} to check whether
the maximal ideal~$\m_i$ of~$R$ has maximal separator degree
for $i=1,2$. Note that $V=(x^3)$ represents a $K$-basis of $F_3 R/F_2 R$
in Step~(3). Moreover, we have $L_1 = R/\m_1= K\oplus Kx$
and $L_2 = R/\m_2= K$. 

First we consider the case $i=1$. We have $\ell_1=2$  and
a $K$-basis of~$L_1$ is given by $(e_1,e_2)=(1,x)$.
Moreover, a $K$-basis $B_1$ of~$S_1$ in Step~(2) is given by 
$B_1=(f_1,f_2)=(-x^2,-x^3)$, and we have $m_1=2$.
Hence we have
$$
\begin{aligned}
L(e_1f_1)&= L(-x^2) =0, &\quad
L(e_1f_2)&= L(-x^3) =-x^3,\; & \\
L(e_2f_1)&= L(-x^3) =-x^3, &\quad
L(e_2f_2)&= L(-x^4) = L(x^2) = 0, &
\end{aligned}
$$
and consequently the matrix $M_1$ in Step~(4) is
$$
M_1= \begin{pmatrix}
  0 &-1 \\
  -1& 0
\end{pmatrix}.
$$
In particular, we have $\rank (M_1)= 2 = m_1$ in Step~(5).
Therefore the maximal ideal~$\m_1$ of~$R$ has maximal separator
degree~$3$. 

Similarly, in the case $i=2$, we have $\ell_2=1$ and
a $K$-basis of~$L_2$ is given by $(e_1)=(1)$. Moreover,
a $K$-basis $B_2$ of~$S_2$ in Step~(2) is given by
$B_2=(f_1,f_2)=(x^3 + x, y^2)$, and thus $m_2=2$.
Consequently, we get
$$
L(e_1f_1)= L(x^3+x) =x^3,\quad
L(e_1f_2)= L(y^2) =0,
$$
which implies that the matrix~$M_2$ in Step~(4) is
$M_2= \begin{pmatrix}
 1 & 0
\end{pmatrix}$.
It follows that $\rank(M_2)=1 <2 =m_2$.
Hence the maximal ideal~$\m_2$ of~$R$ does not
have maximal separator degree.
In fact, we have $\sepdeg(\m_2)=2<3 = \ri(R)$.
\end{example}

Of course, by running the preceding algorithm for $i=1,\dots,s$,
we can check whether~$R$ has the Cayley-Bacharach property.
In the next section we construct another algorithm 
for this purpose which uses the
canonical module of~$R$.

\bigbreak
%
%

\section{The Canonical Module and the Cayley-Bacharach Property}\label{sec4}

In this section we continue to use the notation introduced above.
In particular, we let $R=P/I$ be a 0-dimensional affine $K$-algebra.
Versatile tools to study the ring~$R$ are its canonical module~$\omega_R$
and the affine Hilbert function of~$\omega_R$ which we recall now
(see also~\cite{KR3}, Section 4.5).

\begin{definition}
Let $R=P/I$ be a 0-dimensional affine $K$-algebra.
\begin{enumerate}
\item[(a)] If we equip the $K$-vector space
$\omega_R = \Hom_K(R,K)$ with the $R$-module structure
defined by $f\cdot \phi (g)=\phi(fg)$ for $f,g\in R$ and
$\phi\in\omega_R$, we obtain the {\bf canonical module}
of~$R$.

\item[(b)] For every $i\in\mathbb{Z}$, let
$G_i\omega_R = \{\phi\in\omega_R \mid \phi(F_{-i-1}R)=0\}$.
Then the family $\mathcal{G}= (G_i\omega_R)_{i\in\mathbb{Z}}$
is a $\mathbb{Z}$-filtration
of~$\omega_R$ which we call the {\bf degree filtration}
of~$\omega_R$.

\item[(c)] The map $\HF^a_{\omega_R}:\, \mathbb{Z} \longrightarrow
\mathbb{Z}$ defined by $\HF^a_{\omega_R}(i)= \dim_K(G_i\omega_R)$
for all $i\in\mathbb{Z}$ is called the {\bf affine Hilbert function}
of~$\omega_R$.
\end{enumerate}
\end{definition}

The following proposition collects some properties of the 
degree filtration and the affine Hilbert
function of~$\omega_R$.

\begin{proposition}\label{HFomegaprops}
Let $R=P/I$ be a 0-dimensional affine $K$-algebra.
\begin{enumerate}
\item[(a)] The degree filtration of~$\omega_R$ is increasing,
i.e., we have $G_i\omega_R \subseteq G_j\omega_R$
for $i\le j$. In particular, the affine Hilbert function of~$\omega_R$
is non-decreasing.

\item[(b)] The module $\omega_R$ is a filtered $R$-module, i.e.,
we have $F_i R \cdot G_j\omega_R \subseteq G_{i+j}\omega_R$
for all $i,j\in\mathbb{Z}$.

\item[(c)] For $i\le -\ri(R)-1$, we have $G_i\omega_R=\{0\}$,
and for $i\ge 0$, we have $G_i\omega_R=\omega_R$.
In particular, the filtration $\mathcal{G}$ is exhaustive.

\item[(d)] For every $i\in\mathbb{Z}$, we have
$$
\HF^a_{\omega_R} (i) \;=\; \dim_K(R) - \HF^a_R(-i-1)
$$
In particular, the regularity index of~$\omega_R$ satisfies
$\ri(\omega_R)=0$ and its homogeneous component
of lowest degree has dimension $\HF^a_{\omega_R}(-\ri(R))
=\Delta_R$.
\end{enumerate}
\end{proposition}

\begin{proof}
Claim~(a) follows from the fact that $\phi(F_{-i-1}R)=0$
implies $\phi(F_{-j-1}R)=0$ for $i\le j$.
To check claim~(b), we note that for $f\in F_i R$
and $\phi\in G_j\omega_R$ we have $(f\,\phi)(F_{-i-j-1}R)
= \phi(f\, F_{-i-j-1}R)\subseteq \phi(F_{-j-1}R)=0$,
and therefore we obtain $f\,\phi\in G_{i+j}\omega_R$.

Next we prove~(c). For $i\le -\ri(R)-1$ and $\phi\in G_i\omega_R$
we have $\phi(R)=\phi(F_{\ri(R)}R) \subseteq \phi(F_{-i-1}R) =0$.
Moreover, for $i\ge 0$ we have $\phi\in G_i\omega_R$ if and only if
$\phi(F_{-i-1}R)=0$. Since $F_{-i-1}R=\{0\}$, this holds for all
$\phi\in\omega_R$.

To show~(d) we observe that the condition $\phi(F_{-i-1}R)=0$
defines a $K$-vector subspace of~$\omega_R$ of codimension
$\dim_K(F_{-i-1}R)=\HF^a_R(-i-1)$.
\end{proof}

As for the filtration~$\Fbar$ of~$R$, we can define
the order of an element of~$\omega_R$ with respect to~$\mathcal{G}$.

\begin{definition}
For $\phi\in\omega_R\setminus \{0\}$, the number
$\ord_{\mathcal{G}}(\phi)= \min\{i\in\mathbb{Z}\mid
\phi\in G_i\omega_R\}$ is well-defined. It is called the {\bf order}
of~$\phi$ with respect to~$\mathcal{G}$.
\end{definition}

For computational purposes, the following remark will come in handy.

\begin{remark}
Let $d=\dim_K(R)$, and let $B=(f_1,\dots,f_d)$ be a
degree filtered $K$-basis of~$R$. Then the dual basis
$B^\ast = (f_1^\ast,\dots,f_d^\ast)$ defined by
$f_i^\ast:\, R\longrightarrow K$ with $f_i^\ast(f_j)=\delta_{ij}$
for $i,j=1,\dots,d$ is a degree filtered $K$-basis of~$\omega_R$,
and we have $\ord_{\mathcal{G}}(f_i^\ast)=-\ord_{\mathcal{\Fbar}}(f_i)$
for $i=1,\dots,d$.
\end{remark}

Our next goal is to provide a characterization
of the Cayley-Bacharach property of~$R$ in terms of the
structure of the canonical module~$\omega_R$.
If~$R$ has linear maximal ideals, a straightforward adaptation
of~\cite{Kre2}, Theorem 2.6, achieves this goal.
In fact, our next theorem shows that the hypothesis that~$R$
has linear maximal ideals is not necessary for this
characterization to hold.

\begin{theorem}{\bf (The Canonical Module of a Cayley-Bacharach
Scheme)}\label{CanModCB}\\
Let $R=P/I$ be a 0-dimensional affine $K$-algebra.
Then the following conditions are equivalent.
\begin{enumerate}
\item[(a)] The ring~$R$ has the Cayley-Bacharach property.

\item[(b)] The bilinear map $R\otimes_K G_{-\ri(R)}\omega_R
\longrightarrow \omega_R$ is non-degenerate.

\item[(c)] We have $\Ann_R(G_{-\ri(R)}\omega_R) = \{0\}$.
\end{enumerate}
\end{theorem}

\begin{proof}
First we show that (a) implies~(b). Suppose that there
exists a non-zero element $f\in R$ such that
$f\cdot G_{-\ri(R)}\omega_R = \{0\}$.
Now consider the decomposition into local rings
$\imath:\; R\cong R/\q_1\times \cdots \times R/\q_s$ and
let $\imath(f)=(f_1,\dots,f_s)$. Since $f\ne 0$, there
exists an index $i\in \{1,\dots,s\}$ such that
$f_i\ne 0$. From the fact that~$R/\q_i$ is a local ring and
Lemma 4.5.9.a of~\cite{KR3} we conclude that there exists
an element $g_i\in R/\q_i$ such that $f_ig_i\in\Soc(R/\q_i)
\setminus \{0\}$. By lifting $(0,\dots,0,g_i,0,\dots,0)$,
we therefore get an element $g\in R$
such that $fg$ is contained in the $i$-th socle space
of~$R$. Now the hypothesis that~$R$ has the CBP
implies that in the ideal $\bar{J}=\langle fg\rangle$
there exists an element~$fgh$ with $h\in R$ such that
$\ord_\Fbar(fgh)=\ri(R)$.

\smallskip
Next, we let $\overline{fgh}\ne 0$ be the image of~$fgh$ in
$V= F_{\ri(R)}R / F_{\ri(R)-1}R$. By choosing a complement of
$K\cdot \overline{fgh}$, we find a $K$-linear map $\bar{\phi}:\;
V \longrightarrow K$ such that $\bar{\phi}(\overline{fgh})\ne 0$.
Thus $\bar{\phi}$ lifts to a $K$-linear map $\phi:\; F_{\ri(R)}R
\longrightarrow K$ such that $\phi(F_{\ri(R)-1}R)=\{0\}$ and
$\phi(fgh)\ne 0$. We note that $F_{\ri(R)}R=R$ and conclude that
$\phi\in G_{-\ri(R)}\omega_R$. Hence we obtain
$0\ne \phi(fgh) = (fgh\phi)(1) =(gh\cdot (f\phi))(1)=0$, a
contradiction.

Since~(b) and~(c) are clearly equivalent, it remains to
show that~(c) implies~(a). Suppose that~$R$ does not have 
the CBP. This means that there exists
an $i\in\{1,\dots,s\}$ and a non-zero element $f\in S_i$
in the $i$-th socle space~$S_i$ of~$R$ such that $\ord_\Fbar(fg)\le
\ri(R)-1$ for every element $g\in R$.

Hence we have $\phi(fg)=0$ for every $g\in R$
and every $\phi\in G_{-\ri(R)}\omega_R$.
Consequently, we have $(f\phi)(g)=
\phi(fg)=0$ for every $g\in R$, and hence
$f\phi=0$ for every $\phi\in G_{-\ri(R)}\omega_R$.
This contradicts~(c).
\end{proof}

This theorem can be turned into an algorithm
for checking the Cayley-Bacharach property
as follows. For the notion of a block column matrix 
we refer to~\cite{KR3}, Definition 2.2.16.
We recall that, for a matrix $W \in \Mat_{r,s}(K)$,
the $K$-vector  subspace of $K^s$ given by 
$\Ker(W) = \{(c_1, \dots, c_s)\in K^s \mid W\cdot 
(c_1, \dots, c_s)\tr= 0\}$ is denoted by  $\Ker(W)$.

\begin{algorithm}{\bf (Checking the Cayley-Bacharach Property
Using $\boldsymbol{\omega}_{\mathbf{R}}$)}\label{alg:CheckCBcanmod}\\
Let $R=P/I$ be a 0-dimensional affine $K$-algebra.
Consider the following sequence of instructions.
\begin{enumerate}
\item[(1)] Compute a degree filtered $K$-basis $B=(b_1,\dots,b_d)$
of~$R$. Let $\Delta\ge 1$ be such that $b_{d-\Delta+1},\dots,
b_d$ are the elements of~$B$ of order $\ri(R)$.

\item[(2)] For $i=1,\dots,d$, compute the matrix
$M_B(\vartheta_{b_i})\in\Mat_d(K)$ representing the multiplication 
by~$b_i$ in the basis~$B$.

\item[(3)] For $j=1,\dots,\Delta$, form the 
matrix $V_j \in \Mat_d(K)$ whose $i$-th column is 
the $(d-\Delta+j)$-th column of $M_B(\vartheta_{b_i})\tr$ for 
$i=1,\dots,d$.

\item[(4)] Form the block column matrix
$W=\Col(V_1,\dots,V_\Delta)$ and compute $\Ker(W)$.

\item[(5)] If~$\Ker(W)=\{0\}$, return ${\tt TRUE}$. Otherwise,
return ${\tt FALSE}$.
\end{enumerate}
This is an algorithm which checks whether~$R$ has
the Cayley-Bacharach property and returns the corresponding
Boolean value.
\end{algorithm}

\begin{proof}
Clearly, the algorithm is finite, so that it remains to prove correctness.
According to Theorem~\ref{CanModCB}, to check the CBP
of~$R$, we have to check whether
any non-zero element of~$R$ annihilates $G_{-\ri(R)}\omega_R$.
Since the dual basis $B^\ast=(b_1^\ast,\dots,b_d^\ast)$ is a $K$-basis
of~$\omega_R$, this means that we have to check whether any
$K$-linear combination $a_1 b_1+\cdots +a_d b_d$ annihilates
$b_{d-\Delta+1}^\ast, \dots, b_d^\ast$. Using \cite{KR3}, Remark~4.5.3, 
it is easy to see that the coordinate tuple of $b_i\,b_k^\ast$
in the basis $B^\ast$ is given by the $k$-th column
of~$M_B(\vartheta_{b_i})\tr$. Hence the $(d-\Delta+j)$-th column 
of $M_B(\vartheta_{b_i})\tr$ 
is the coordinate tuple
of $b_i\,b_{d-\Delta+j}^\ast$ in the basis~$B^\ast$.
Thus the coordinate tuples of all elements of~$R$
which annihilate $b_{d-\Delta+j}^\ast$ are given by the tuples
$(a_1,\dots,a_d)\in K^d$ such that $V_j\cdot(a_1,\dots,a_d)\tr=0$.
Altogether, the vector
space~$\Ker(W)$ contains all coordinate tuples of elements of~$R$
which annihilate $G_{-\ri(R)}\omega_R$. From this the
claim follows.
\end{proof}

Let us check this algorithm in a couple of examples.

\begin{example}\label{needInfinite}
Let $K= \mathbb{F}_2$, let $P = K[x,y]$, let~$I$ be the ideal of~$P$ 
generated by $\{x^2+x,\; y^2+y, \; xy\}$, and let $R = P/I$. The primary decomposition
of~$I$ is $I = \M_1\cap \M_2\cap \M_3$, where $\M_1 = \langle x,\, y \rangle$,
$\M_2 = \langle x,\, y+1\rangle$, and $\M_3 = \langle x+1,\, y \rangle$.
The affine Hilbert function of~$R$ is $(1,3,3,\dots)$, and hence $\ri(R) = 1$.
A degree filtered basis of $R$ is given by the residue classes of $\{1,\; y,\;x\}$.
Thus we have $d=3$ and~$\Delta_R=2$. The two matrices~$V_1$ and~$V_2$ computed in 
Step~(3) of the algorithm are
$$
V_1 = \begin{pmatrix} 0 & 1 & 0\cr 1 & 1 & 0 \cr 0 & 0 & 0 \end{pmatrix} 
\hbox{\quad and\quad } V_2 = \begin{pmatrix} 0 & 0 & 1\cr 0 & 0 & 0 \cr 1 & 0 & 1 \end{pmatrix} 
$$
Since the matrix $W = \Col(V_1,V_2)$ has a trivial kernel, we conclude that~$R$ 
has the~CBP.
\end{example}

\begin{example}\label{CBDeg6}
Let $K = \mathbb{Q}$, let $P = K[x,y,z]$, let~$I$ be the ideal of~$P$ generated by 
$\{z^2 -x +2z, \; xz -2x -y +4z,\; y^2 -x +z,\;  x^2 -yz -4x -4y +8z\}$,
and let $R = P/I$. The primary decomposition
of~$I$ is $I = \M_1\cap\M_2$, where we have $\M_1 = \langle x,\; y,\; z \rangle$
and $\M_2 = \langle z^2 -x +2z,\;  xz -2x -y +4z, \; y^2 -x +z,\;   
x^2 -yz -4x -4y +8z, \; xy -2yz -z -1\rangle$. Here~$\M_1$ and~$\M_2$ are maximal ideals, 
$\M_1$ is a linear maximal ideal, and~$\M_2$ corresponds to a residue field 
extension $K\subset L_2$ of degree~5. The affine Hilbert function 
of~$R$ is $(1,4,6,6,\dots)$, and hence $\ri(R) = 2$.
A degree filtered $K$-basis of~$R$ is given by the residue classes of 
$\{1,\; z,\; y,\;x,\; yz,\; xy\}$.
Thus we have $d=6$ and $\Delta_R=2$. The two matrices~$V_1$ and~$V_2$ computed in 
Step~(3) of the algorithm are
$$
V_1 = \left(
\begin{array}{rrrrrr}
0  &  0 &  \ 0 &  0 &  1 &  0\cr 
0  & 0 &  1 &  0 & -2 &-4 \cr
 0  &  1 &  0 &  0 &  0 &  1 \cr
 0  &  0 &  0 &  1 & -4 & -8 \cr
{-1} & -2 &  0 & -4 &  1 &  1 \cr
 0  & -4  &  1 & -8 &  1 & -2
 \end{array} 
 \right)
\hbox{\quad and\quad } 
V_2 = \begin{pmatrix} 
0 & 0 & 0 & 0 & 0 &1\cr 
0 & 0 & 0 & 0 & 1 &2 \cr
0 & 0 & 0 & 1 & 0 &0 \cr
0 & 0 & 1 & 0 & 2 &4 \cr
0 & 1 & 0 & 2 & 0 &1 \cr
1 & 2 & 0 & 4 & 1 &5 
\end{pmatrix} 
$$
Since the matrix $W = \Col(V_1,V_2)$ has a trivial kernel, we conclude that~$R$ 
has the~CBP.
\end{example}

An interesting consequence of Algorithm~\ref{alg:CheckCBcanmod} 
is that the Cayley-Bacharach property of~$R$ is invariant under an
extension of base fields, as the following corollary shows.

\begin{corollary}\label{CBindependentonK}
Let $K$ be a field, let $K\subset L$ be a field extension,
and let $R$ be a 0-dimensional affine $K$-algebra.
Then the following conditions are equivalent.
\begin{enumerate}
\item[(a)] The ring $R$ has the Cayley-Bacharach property.

\item[(b)] The ring $R\otimes_K L$ has the Cayley-Bacharach property.
\end{enumerate}
\end{corollary}

\begin{proof} Given a field extension $K\subset L$
and a degree filtered $K$-basis $B=(b_1,\dots,b_d)$ of~$R$,
the tuple~$B$ is also a degree filtered $L$-basis of $R\otimes_K L$.
Moreover, the corresponding multiplication matrices $M_B(\vartheta_{b_i})$ 
do not change under this field extension. 
Hence the matrix~$W$ computed in Step~4 of Algorithm~\ref{alg:CheckCBcanmod}
does not depend on the field extension, and the claim follows.
\end{proof}

Let us show an example which illustrates this corollary.

\begin{example}\label{CBBurj}
Let $K= \mathbb{Q}$, let $P=K[x,y]$,
and let $R=P/I$ where
$I = \M_1 \cap \M_2$ with
$\M_1 = \langle x^5-x-2, y-x^3 \rangle$
and $\M_2 = \langle x,y\rangle$.
Note that~$I$ is a radical ideal.

First we use Algorithm~\ref{alg:CheckCBcanmod} to check whether~$R$ 
has the~CBP.
The reduced Gr\"{o}bner basis of the ideal~$I$ with respect to {\tt DegRevLex}
is given by $\{x^2 - y^2 + 2x,\allowbreak xy^2 - 2y^2 + 4x - y,\,
y^4 - xy - 4y^2 + 8x - 4y\}$. Consequently,
the residue classes of the elements in $B=(1, y, x, y^2, xy, y^3)$
form a degree filtered $K$-basis of~$R$.
We have $d= 6$, $\Delta_R = 1$, and the corresponding matrix~$W$ is
$$
W = \begin{pmatrix} 
0 & 0 & 0 & 0 & 0 &1\cr 
0 & 0 & 0 & 1 & 0 &0 \cr
0 & 0 & 0 & 0 & 1 &2 \cr
0 & 1 & 0 & 0 & 2 &4 \cr
0 & 0 & 1 & 2 & 0 &1 \cr
1 & 0 & 2 & 4 & 1 &6
\end{pmatrix} 
$$
Next we let~$L$ be the splitting field of the polynomial
$x^5-x-2$ over $K$.
The reduced Gr\"{o}bner basis of $I\, L[x,y]$
with respect to {\tt DegRevLex} is again
$\{x^2 -y^2 +2x,\allowbreak  xy^2 -2y^2 +4x -y,\;
y^4 -xy -4y^2 +8x -4y \}$ (see~\cite{KR1}, Lemma~2.4.16), 
and the residue classes of the elements in~$B$
form a degree filtered $L$-basis of~$R\otimes_K L$.
Clearly, the multiplication matrices $M_B(\vartheta_{b_i})\in\Mat_d(K)$
express $b_i\,b_j$ as a linear combination of $b_1,\dots, b_d$
both over~$K$ and~$L$.
Thus the matrix~$W$ agrees with the matrix constructed
by applying Algorithm~\ref{alg:CheckCBcanmod} to 
$R\otimes_K L = L[x,y]/I\,L[x,y]$.

Altogether, the ring~$R$ has the~CBP
if and only if we have $\Ker(W)=\{0\}$, and this holds if and only if
$\Ker(W)\otimes_K L=\{0\}$, i.e., if and only if
$R\otimes_K L$ has the~CBP, as expected.
In this example $R$ and $R\otimes_K L$ have the~CBP.
Notice that the scheme $\mathbb{X} = \Spec(P/I)$ consists of
two reduced points, while $\mathbb{X}_L = \Spec(L[x,y]/I\,L[x,y])$
consists of six reduced points.
\end{example}

\bigbreak
%
%

\section{Locally Gorenstein Rings and the Cayley-Bacharach Property}
\label{sec5}

As in the preceding sections, we
continue to let~$R=P/I$ be a 0-dimensional affine $K$-algebra.
The canonical module~$\omega_R$ of~$R$ can be used to characterize 
the property of~$R$ to be a locally Gorenstein ring. Notice that this
property is sometimes simply called Gorenstein, but here we
want to emphasize the distiction with the notion of strict
Gorenstein rings which will be considered later 
(see Definition~\ref{defstrictGor}).

\begin{definition}
Let $R$ be a 0-dimensional affine $K$-algebra,
and let $\q_1,\dots,\q_s$ be the primary components
of the zero ideal of~$R$. We say that~$R$ is
{\bf locally Gorenstein} if $R/\q_i$ is a Gorenstein local ring for
$i=1,\dots,s$. 
\end{definition}

Recalling the decomposition of~$R$ into local rings
$R\cong R/\q_1 \times\cdots \times R/\q_s$, we see that~$R$
is locally Gorenstein if and only if its local factors are Gorenstein.
Clearly, a field~$R$ is Gorenstein.  Every reduced
0-dimensional affine $K$-al\-ge\-bra~$R$ is a locally Gorenstein
ring, as we can see by applying the isomorphism
$R\cong R/\m_1 \times \cdots\times R/\m_s$.

The following extension of~\cite{KR3}, Theorem 4.5.21,
will allow us to check the Gorenstein property of~$R$
in a nice way.

\begin{theorem}{\bf (The Canonical Module of 
a Locally Gorenstein Ring)}\label{CharGor}\\
Let $K$ be a field, and let~$R$ be a 0-dimensional affine $K$-algebra.
Then the following conditions are equivalent.
\begin{enumerate}
\item[(a)] The ring $R$ is locally Gorenstein.

\item[(b)] The canonical module $\omega_R$ is a cyclic $R$-module, i.e.,
there exists an element $\phi\in \omega_R$ such that  $\omega_R = \langle \phi\rangle$.

\item[(c)] There exists an element $\phi\in \omega_R$ such that 
$\Ann_R(\phi) = \{0\}$.
\end{enumerate}
\end{theorem}

\begin{proof}
The equivalence between~(a) and~(b) is part of Theorem~4.5.21 of~\cite{KR3}.
To show that~(b) implies~(c), we let $\phi\in\omega_R$ be a generator
of~$\omega_R$. Then multiplication by~$\phi$ induces an $R$-module
isomorphism $R/\Ann_R(\phi) \longrightarrow \omega_R$, and from
$\dim_K(R)=\dim_K(\omega_R)$ we conclude that $\Ann_R(\phi)=\{0\}$.

Conversely, given an element $\phi\in\omega_R$ with
$\Ann_R(\phi) = \{0\}$, multiplication by~$\phi$ induces
an $R$-module isomorphism $R\longrightarrow \langle \phi\rangle$,
and again $\dim_K(R)=\dim_K(\omega_R)$ implies that
$\langle\phi\rangle = \omega_R$.
\end{proof}

Note that the equivalence between~(b) and~(c) in this theorem can be viewed as
a special case of the equivalence proved in Lemma 4.5.8 of~\cite{KR3}.
As a consequence of the theorem we obtain the following well-known
result. 

\begin{corollary}\label{GorExtends}
Let $K$ be a field, let $K\subset L$ be a field extension,
and let~$R$ be a 0-dimensional affine $K$-algebra.
Then the following conditions are equivalent.
\begin{enumerate}
\item[(a)] The ring $R$ is locally  Gorenstein.

\item[(b)] The ring $R\otimes_K L$ is locally Gorenstein.
\end{enumerate}
\end{corollary}

\begin{proof} By part (b) of the theorem, we have to show that 
the canonical module~$\omega_R$ is a cyclic $R$-module if and only if
$\omega_{R\otimes_K L}$ is a cyclic $R\otimes_K L$-module.
This equivalence follows from~\cite{KR3}, Theorem 3.6.4.a.
\end{proof}

By including~\cite{KR3}, Algorithm 3.1.4, we can now rewrite Algorithm~4.5.22
of~\cite{KR3} as follows.

\begin{algorithm}{\bf (Checking the Locally Gorenstein Property)}\label{alg:gor}\\
Let $R=P/I$ be a 0-dimensional affine $K$-algebra
given by a set of generators of~$I$, and let $d=\dim_K(R)$.
The following instructions define an algorithm which
checks whether~$R$ is a locally Gorenstein ring and returns the corresponding
Boolean value.
\begin{enumerate}
\item[(1)] Compute a tuple of polynomials whose residue classes
$B=(b_1,\dots,b_d)$ form a $K$-basis of~$R$.

\item[(2)] Compute the multiplication matrices $M_B(\theta_{x_1}), \dots,
M_B(\theta_{x_n})$, i.e., the matrices representing the multiplication
maps $\theta_{x_i}:\; R \longrightarrow R$ in the basis~$B$.

\item[(3)] Let $z_1,\dots,z_d$ be new indeterminates,
and let $C \in \Mat_d(K[z_1,\dots,z_d])$ be the matrix
whose columns are
$b_i \big( M_B(\theta_{x_1})\tr, \dots, M_B(\theta_{x_n})\tr \big)
\cdot (z_1,\dots, z_d)\tr$
for $i=1,\dots,d$.

\item[(4)] If $\det(C) \ne 0$ return {\tt TRUE}, otherwise
return {\tt FALSE}.
\end{enumerate}

Moreover, if  the ring $R$ is locally Gorenstein and 
$(c_1,\dots,c_d)\in K^d$ is such that $\det(C(c_1,\dots,c_d))\ne 0$,
then the linear form $c_1b_1^*+\cdots c_db_d^*$ is a generator of~$\omega_R$.
\end{algorithm}

\begin{proof}
By the Cyclicity Test (cf.~\cite{KR3}, Algorithm~3.1.4), Step~(4)
checks whether the module whose multiplication matrices
are the matrices $M_B(\theta_{x_i})\tr$ is a cyclic \hbox{$R$-module}.
By~\cite{KR3}, Remark 4.5.3, these are exactly the multiplication
matrices of the canonical module. So, the algorithm checks whether
$\omega_R$ is a cyclic $R$-module and returns the correct
answer according to Theorem~\ref{CharGor}.
The additional claim follows from~\cite{KR3},  Algorithm 3.1.4.
\end{proof}

Let us have a look at an example which illustrates this algorithm.

\begin{example}\label{GorCBNonRedNonLin}
Let $K= \mathbb{Q}$, let $P = K[x,y,z]$, and let $R=P/I$,
where $I$ is the ideal of~$P$ generated by
\[
\begin{array}{l}
\big\{\, x^{2} -18xy + 43y^{2} + 12xz -\tfrac{170}{3}yz
+ \tfrac{218}{3}z^{2} -4x + \tfrac{340}{3}y -216z + \tfrac{166}{3},\\
\ \, xy^{2} -3xy - \tfrac{4}{9}y^{2} + xz - \tfrac{32}{27}yz
-\tfrac{28}{27}z^{2}+\tfrac{64}{27}y+\tfrac{28}{9}z-\tfrac{32}{27},\\
\ \, y^{3} -\tfrac{17}{9}y^{2} + \tfrac{17}{27}yz -\tfrac{2}{27}z^{2}
-\tfrac{88}{27}y + \tfrac{20}{9}z -\tfrac{10}{27},\\
\ \, y^{2}z -\tfrac{10}{9}y^{2} -\tfrac{17}{27}yz + \tfrac{83}{27}z^{2}
+ \tfrac{34}{27}y -\tfrac{74}{9}z + \tfrac{64}{27},\\
\ \, z^{3} + \tfrac{2}{9}y^{2} -\tfrac{11}{27}yz -\tfrac{40}{27}z^{2}
+ \tfrac{22}{27}y -\tfrac{14}{9}z + \tfrac{16}{27},\\
\ \, xz^{2} - xy -\tfrac{1}{9}y^{2} -\tfrac{8}{27}yz -\tfrac{7}{27}z^{2}
+ \tfrac{16}{27}y + \tfrac{7}{9}z -\tfrac{8}{27},\\
\ \, yz^{2} + \tfrac{2}{9}y^{2} -\tfrac{38}{27}yz -\tfrac{67}{27}z^{2}
-\tfrac{32}{27}y + \tfrac{49}{9}z -\tfrac{38}{27},\\
\ \, xyz -\tfrac{1}{9}y^{2} -3xz -\tfrac{8}{27}yz -\tfrac{7}{27}z^{2}
+ x + \tfrac{16}{27}y + \tfrac{7}{9}z -\tfrac{8}{27} \,\big\}.
\end{array}
\]
Let us check whether $R$ is locally Gorenstein  or not.
Note that the given generating set is the reduced
Gr\"{o}bner basis of~$I$ with respect to \texttt{DegRevLex}.
So, a degree filtered $K$-basis $B$ of~$R$ is given by the residue
classes of the elements in the tuple
$(1,\; z,\; y,\; x,\; z^2,\; yz,\; xz,\; y^2,\; xy)$.

As the computation of the determinant of the matrix $C \in K[z_1,\dots,z_9]$
of size $9\times 9$ in Step~(4) of Algorithm~\ref{alg:gor}
is quite demanding, we substitute in~$C$ the 
indeterminates $(z_1,\dots, z_9)$ by
the numbers $\lambda =(1,-3,-1,2,4,-1,-1,1,3)$ and get
$$
C_\lambda =
\left(
\begin{array}{rrrrrrrrr}
  1\! & -3 & -1 & 2 & 4 & -1 & -1 & 1 & 3 \\
  -3\! & 4 & -1 & -1 & \frac{23}{27} & \frac{671}{27}
  & \frac{191}{27} & -\frac{1015}{27} & -\frac{25}{27} \\[3pt]
  -1\! & -1 & 1 & 3 & \frac{671}{27} & -\frac{1015}{27}
  & -\frac{25}{27} & \frac{178}{27} & \frac{710}{27} \\[3pt]
  2\! & -1 & 3 & -\frac{2719}{3} & \frac{191}{27}
  & -\frac{25}{27} & \frac{108017}{27} & \frac{710}{27}
  & -\frac{107924}{27} \\[3pt]
  4\! & \frac{23}{27} & \frac{671}{27} & \frac{191}{27}
  & \frac{257}{9} & \frac{493}{27} & -\frac{25}{27}
  & \frac{338}{27} & \frac{200}{9} \\[3pt]
  -1\! & \frac{671}{27} & \!-\frac{1015}{27} & -\frac{25}{27}
  & \frac{493}{27} & \frac{338}{27} & \frac{200}{9}
  & -\frac{2696}{27} & -\frac{266}{27} \\[3pt]
  -1\! & \frac{191}{27} & -\frac{25}{27} & \frac{108017}{27}
  & -\frac{25}{27} & \frac{200}{9} & -\frac{35938}{9}
  & -\frac{266}{27} & \frac{348632}{27} \\[3pt]
  1\! & \!-\frac{1015}{27} & \frac{178}{27} & \frac{710}{27}
  & \frac{338}{27} & -\frac{2696}{27} & -\frac{266}{27}
  & \frac{1163}{27} & \frac{1715}{27} \\[3pt]
  3\! & -\frac{25}{27} & \frac{710}{27} & -\frac{107924}{27}
  & \frac{200}{9} & -\frac{266}{27} & \frac{348632}{27}
  & \frac{1715}{27} & -\frac{143783}{9}
\end{array}
\right)
$$
Since we have 
$\det(C_\lambda)= \frac{114824810760065082500447360}{10460353203} \ne 0$,
we know that $\det(C)\ne 0$, and hence the ring~$R$ is locally Gorenstein.
\end{example}

In the remaining part of this section we construct
an algorithm for checking whether~$R$ is locally Gorenstein and has
the Cayley-Bacharach property. 
Given a field extension $K \subseteq L$, for simplicity we write
$R_L$ to denote $R\otimes_K L$.
The following theorem is the key result.

\begin{theorem}{\bf (Locally Gorenstein Rings 
Having the Cayley-Bacharach Property)}\label{CharGorCB}\\
Let~$K$ be a field, and let $R=P/I$ be a 0-dimensional
affine $K$-algebra. Then the following conditions are equivalent.
\begin{enumerate}
\item[(a)] There exists a field extension  
$K\subseteq L$ and an element $\phi\in\omega_{R_L}$
such that we have $\ord_{\mathcal{G}}(\phi)=-\ri(R_L)$ and
${\Ann_{R_L} (\phi)=\{0\}}$.

\item[(b)] The ring~$R$ is locally Gorenstein and has the Cayley-Bacharach
property.
\end{enumerate}

If~$K$ is infinite, then $L=K$ satisfies Condition~(a).
\end{theorem}

\begin{proof} 
By Corollary~\ref{GorExtends}, we know that~$R$ is locally Gorenstein 
if and only if~$R_L$ is locally Gorenstein.
We also know by Corollary~\ref{CBindependentonK} that~$R$ 
has the CBP if and only if~$R_L$ has the~CBP. 
Together with $\ri(R) = \ri(R_L)$ this implies that 
we may assume that~$L=K$ and that this is an infinite field.
 
First we show that~(a) implies~(b). By Theorem~\ref{CharGor},
the ring~$R$ is locally Gorenstein.
Let $i\in\{1,\dots,s\}$, and let $f_i\in R$ be a separator
for~$\m_i$. Then $f_i\phi\ne 0$ implies that
there exists an element $g\in R$ such that $(f_i\phi)(g)\ne 0$.
In particular, we have $f_i g\ne 0$, so that passing to the
residue classes in $R_i=R/\q_i$ shows that $\bar{g}$ is a unit
in~$R_i$. Consequently, also $f_ig$ is a separator for~$\m_i$,
and $\phi(f_ig)\ne 0$ together with $\ord_{\mathcal{G}}(\phi)=
-\ri(R)$ yields $\ord_\Fbar(f_ig)=\ri(R)$.
This shows that, for the every minimal $\Q_i$-divisor~$J_i$ of~$I$,
there exists a separator $f_ig\in \bar{J}_i$ of order $\ri(R)$.
Thus the ring~$R$ has the~CBP.

Now we prove that~(b) implies~(a).
Since~$R$ is a locally Gorenstein ring, we have the equality 
$\dim_{L_i}(\Soc(R_i))=1$
for $i=1,\dots,s$. Hence, for every $i\in\{1,\dots,s\}$,
there exists a unique minimal $\Q_i$-divisor~$J_i$ of~$I$
whose image~$\bar{J}_i$ in~$R$ is the preimage of $\Soc(R_i)$
under the decomposition~$\imath$ of~$R$ into local rings. 
Using the fact that~$R$ has the~CBP,
we therefore find for every $i\in\{1,\dots,s\}$ a
separator $f_i\in \bar{J}_i$ for~$\m_i$ such that
$\ord_\Fbar(f_i)=\ri(R)$.

Now we let $\rho=\ri(R)$ and consider the $K$-vector
space $F_\rho R/F_{\rho-1}R$. Using the hypothesis that~$K$
is infinite and the fact that the residue classes $\bar f_1,\dots,\bar f_s$
are all non-zero, we find a $K$-linear map $\bar\phi:\;
F_\rho R/ F_{\rho-1}R \longrightarrow K$ such that $\bar\phi(\bar f_i)\ne 0$
for $i=1,\dots,s$. By composing~$\bar\phi$ with the canonical epimorphism,
we therefore get a $K$-linear map $\phi:\; F_\rho R \longrightarrow K$
such that $\phi(f_i)\ne 0$ for $i=1,\dots,s$. Clearly, the map~$\phi$
is an element of~$G_{-\rho}\omega_R$, and hence $\ord_{\mathcal{G}}(\phi)=
-\ri(R)$.

It remains to prove that we have $\Ann_R(\phi)=\{0\}$.
Assume that $g\cdot \phi=0$ for some element
$g\in R\setminus \{0\}$. Let $\imath(g)=(g_1,\dots,g_s)$
with $g_i\in R_i$ for $i=1,\dots,s$. Since $R_i$ is a local Gorenstein ring,
there exists an element $h_i\in R_i$ such that we have $h_i g_i\in\Soc(R_i)$
(see for instance~\cite{KR3}, Lemma 4.5.9.a).
Furthermore, as we know $\dim_{L_i}(\Soc(R_i))=1$, we find a unit
$u_i\in R_i$ such that $g_ih_i=u_i \tilde{f}_i$ for the residue class 
$\tilde{f}_i =f_i+\q_i \in R_i$.
By defining $k_i\in R$ to be the preimage of 
$(0,\dots,0,h_iu_i^{-1},0,\dots,0)$ under~$\imath$, we then 
get $g k_i=f_i$. This implies
$0=g\phi(k_i) = \phi(g k_i)=\phi(f_i)\ne 0$, a contradiction.
\end{proof}

The existence of an element $\phi\in G_{-\ri(R)}\omega_R$
such that $\Ann_R(\phi)=\{0\}$
has the following effect on the affine Hilbert 
function of~$R$.

\begin{corollary}\label{HFinequal}
Let $K$ be a field, and let~$R$ be a 0-dimensional 
affine $K$-algebra. If there exists an element 
$\phi\in G_{-\ri(R)}\omega_R$ such that $\Ann_R(\phi)=\{0\}$, 
then we have
$$
\HF^a_R(i) + \HF^a_R(\ri(R)-1-i) \;\le\; \dim_K(R)
$$
for $i=0,\dots,\ri(R)-1$.

In particular, these inequalities hold if~$K$ is infinite
and~$R$ is a locally Gorenstein ring with the Cayley-Bacharach property.
\end{corollary}

\begin{proof}
By the hypothesis, we have
$\dim_K (F_iR \cdot \phi)= \dim_K(F_iR) = \HF^a_R(i)$ 
for all $i\in\mathbb{Z}$.
Thus the claim follows from $(F_iR)\cdot\phi \subseteq 
G_{-\ri(R)+i}\omega_R$ and Proposition~\ref{HFomegaprops}.d.
\end{proof}

The preceding theorem allows us to generalize Algorithm~4.6.21 in~\cite{KR3}
by dropping the hypothesis that the maximal ideals of~$R$ are linear.
To emphasize the analogy with Algorithm~\ref{alg:gor},
we first rewrite~\cite{KR3}, Lemma~4.6.20 as follows.

\begin{lemma}\label{usingmatrices}
Let $R=P/I$ be a 0-dimensional affine $K$-algebra, let $\phi\in\omega_R$,
and let $B=(b_1,\dots,b_d)$ be a $K$-basis of~$R$. 
We write $\phi=c_1 b_1^\ast + \cdots
+c_d b_d^\ast$ with $c_1,\dots,c_d\in K$.
\begin{enumerate}
\item[(a)] For $g\in R$, we have $g\,\phi=0$ in~$\omega_R$
if and only if we have the equality
$M_B(\theta_g)\tr \cdot (c_1,\dots,c_d)\tr  =0$.

\item[(b)] Let $\Lambda_c\in \Mat_d(K)$ be the matrix
whose $i$-th column is given by the product
$M_B(\theta_{b_i})\tr \cdot(c_1,\dots,c_d)\tr$
for $i\in\{1,\dots,d\}$.
Then we have $\Ann_R(\phi)=\{ 0\}$
if and only if $\det(\Lambda_c)\ne 0$.
\end{enumerate}
\end{lemma}

\begin{proof}
It suffices to note that the matrices in~(a) and~(b)
are the transposes of the corresponding matrices 
in~\cite{KR3}, Lemma~4.6.20.
\end{proof}

Now we are ready to formulate the desired algorithm.

\begin{algorithm}{\bf (Checking Locally Gorenstein Rings 
with the Cayley-Bacharach Property)}\label{alg:GorCB}\\
Let $K$ be a field, and let
$R=P/I$ be a 0-dimensional affine $K$-algebra.
Consider the following sequence of instructions.
\begin{enumerate}
\item[(1)] Compute a degree filtered $K$-basis $B=(b_1,\dots, b_d)$
of the ring~$R$. Let $\Delta\ge 1$ be such that
$b_{d-\Delta+1},b_{d-\Delta+2},\dots, b_d$
are those elements in~$B$ whose order is $\ri(R)$.

\item[(2)] Compute the multiplication matrices $M_B(\theta_{x_1}), \dots,
M_B(\theta_{x_n})$.

\item[(3)] Let $z_{d-\Delta+1},\dots, z_d$ be indeterminates. Let
$C_0 \in \Mat_d(K[z_{d-\Delta+1},\dots,z_d])$ be the matrix
such that, for $i=1,\dots,d$, its $i$-th column is given by
$b_i \big( M_B(\theta_{x_1})\tr, \dots, M_B(\theta_{x_n})\tr \big)
\cdot  (0,\dots,0,z_{d-\Delta+1},\dots,z_d)\tr$.

\item[(4)] If $\det(C_0)\ne 0$, return ${\tt TRUE}$. Otherwise,
return ${\tt FALSE}$.
\end{enumerate}
This is an algorithm which checks whether~$R$ is a locally Gorenstein 
ring which has the Cayley-Bacharach property and returns the 
corresponding Boolean value.
\end{algorithm}

\begin{proof}
Notice that the matrix $C_0$ in this algorithm is obtained from the
matrix $C$ of Algorithm~\ref{alg:gor} by replacing $z_1, \dots,
z_{d-\Delta}$ with~0. Hence the condition ${\det(C_0)\ne 0}$
implies $\det(C)\ne 0$, and thus~$R$ is a 
locally Gorenstein ring in this case.
To prove that~$R$ has the~CBP, we may assume
by Corollary~\ref{CBindependentonK} that the field~$K$ is infinite.
By Theorem~\ref{CharGorCB}, we have to show that there exists
an element $\phi\in\omega_R$ of order $-\ri(R)$ such that 
$\Ann_R(\phi)=\{0\}$. Since the elements $b^\ast_{d-\Delta+1},\dots,
b^\ast_d$ form a $K$-basis of $G_{-\ri(R)}\omega_R$,
the lemma shows that we have to find a non-zero
tuple $(a_{d-\Delta+1},\dots,a_d)\in K^\Delta$
such that $C_0(a_{d-\Delta+1},\dots,a_d)\ne 0$.
When $\det(C_0)\ne 0$ and~$K$ is infinite, this is
clearly possible.

To conclude the proof, we show that if $\det(C_0)=0$ and~$R$ is 
a locally Gorenstein ring, 
then~$R$ does not have the~CBP.
We observe that the basis~$B$, the matrices of~$\theta_{x_1}, \dots, \theta_{x_n}$,
and hence the matrix~$C_0$ do not change if we extend the base field. 
Therefore we may assume that~$K$ is infinite.
Now the equality $\det(C_0) = 0$  implies that there is no linear form 
$\phi \in \omega_R$ such that 
$\ord_{\mathcal{G}}(\phi)= -\ri(R)$ and $\Ann_R(\phi) = \{0\}$. Thus
the claim follows from Theorem~\ref{CharGorCB}.
\end{proof}

Using Example~\ref{needInfinite}, we can see that there are cases where
a non-trivial field extension is required to satisfy condition~(a) 
of Theorem~\ref{CharGorCB}.

\begin{example}\label{needInfinite-continued}
In Example~\ref{needInfinite} we have
$$
C_0 = \begin{pmatrix}
0 & z_2 & z_3\\
z_2 & z_2 & 0\\
z_3 & 0 & z_3
\end{pmatrix}
$$
and $\det(C_0) = z_2z_3(z_2+z_3)$. This shows that, notwithstanding the fact that 
$\det(C_0)\ne0$, all pairs in $\mathbb{F}_2^2$ are zeros of~$\det(C_0)$.
Therefore the ring~$R$ is locally Gorenstein, has the~CBP, 
satisfies $\ri(R) = 1$, but every element $\phi\in \omega_R$ such that 
$\ord_{\mathcal{G}}(\phi)=-\ri(R) = -1$ has a non-trivial annihilator.

Now we let $L= \mathbb F_2[a]/\langle a^2+a+1\rangle$, and let~$\bar{a}$
be the residue class of~$a$ in~$L$. Then $(1+\bar{a}, 1)$ is not a zero 
of~$\det(C_0)$. Hence, for the element $\phi = (1+\bar{a})x^* + y^* \in \omega_{R_L}$, 
we have $\Ann_{R_L}(\phi) = \{0\}$.
\end{example}

Next we present a classical example of a reduced scheme
which does not have the Cayley-Bacharach property and check it
using Algorithm~\ref{alg:GorCB}.

\begin{example}\label{NonCBstandard}
Let $K= \mathbb{Q}$, let $P=K[x,y]$, let~$I$ be the ideal
$I = \langle xy,\; {y^2 -y}, \allowbreak{x^3 -x} \rangle$
in~$P$, and let $R=P/I$.
The primary decomposition of~$I$ is given by
$I = \M_1\cap \M_2\cap \M_3 \cap \M_4$, where
$\M_1 = \langle x+1,\; y\rangle$,
$\M_2 =\langle x,\; y \rangle$,
$\M_3 =\langle x,\; y-1\rangle$,
and $\M_4 =\langle x-1,\; y\rangle$.
Since~$R$ is reduced, it is also a locally Gorenstein ring.

Now we check whether~$R$ has the~CBP.
The set $\{xy, \allowbreak{y^2 -y},\allowbreak x^3 -x\}$ is the reduced
Gr\"{o}bner basis of $I$ with respect to {\tt DegRevLex}.
So, a degree filtered $K$-basis $B$ of~$R$ is given by
the residue classes of the elements in $(1,\; y,\; x, \; x^2)$.
In particular, we have $d=4$ and $\Delta_R=1$.
The corresponding matrix~$C_0$~is
$$ C_0 =
\begin{pmatrix}
0 & 0 & 0 & z_4 \\
  0 & 0 & 0 & 0 \\
  0 & 0 & z_4 & 0 \\
  z_4 & 0 & 0 & z_4
\end{pmatrix}
$$
and its determinant is $\det(C_0) = 0$.
Thus we conclude that~$R$ is a locally Gorenstein ring which
does not have the~CBP.
\end{example}

In the last part of this section we consider the case $\Delta_R=1$
more thoroughly. The following remark provides an important connection
between Algorithms~\ref{alg:CheckCBcanmod} and~\ref{alg:GorCB}.

\begin{remark}\label{C0islincomb}
Let $C_0$ be the matrix  computed in Step~(3) of Algorithm~\ref{alg:GorCB}, 
and let $V_1,\dots,V_\Delta$  be the matrices computed in Step~(3) 
of Algorithm~\ref{alg:CheckCBcanmod}. Then the construction of these matrices
implies that we have $C_0 = z_{d-\Delta+1}V_1 + \cdots + z_d V_\Delta$.
\end{remark}

Now Algorithms~\ref{alg:CheckCBcanmod} and~\ref{alg:GorCB} yield the 
following characterization of the CBP for rings whose last difference is one.

\begin{corollary}\label{deltaqual1}
In the setting of Algorithm~\ref{alg:GorCB},
assume that $\Delta_R =1$. Then the following conditions are equivalent.
\begin{enumerate}
\item[(a)] The ring $R$ has the Cayley-Bacharach property.

\item[(b)] The ring $R$ is locally Gorenstein and its canonical module is generated by an 
element~$\phi$ such that $\ord_{\mathcal{G}}(\phi) = -\ri(R)$.
\end{enumerate}
\end{corollary}

\begin{proof}
First we observe that if $\Delta_R = 1$ then there is only one matrix~$V_1$ 
in Step~(3) of Algorithm~\ref{alg:CheckCBcanmod}. Consequently, 
Remark~\ref{C0islincomb} says that we have $C_0 = z_d V_1$
in Algorithm~\ref{alg:GorCB}.

Now let us prove that~(a) implies~(b).
Algorithm~\ref{alg:CheckCBcanmod} shows that we have 
$\det(V_1) \ne 0$, and therefore $\det(C_0) \ne 0$. Then the
matrix~$C$ in Step~(3) of Algorithm~\ref{alg:gor} satisfies
$\det(C) \ne 0$ as well, since we already observed in 
the proof of Algorithm~\ref{alg:GorCB} that~$C_0$  is obtained 
from~$C$ by replacing $z_1,\dots,z_{d-1}$ by~$0$. 
Thus Algorithm~\ref{alg:gor} shows that~$R$ is locally Gorenstein. 
Moreover, we note that we have $\det(C_0) = z_d^d\det(V_1)\ne 0$. This implies
$\det(C_0)(1)\ne 0$, and therefore $\det(C)(0,\dots,0,1)\ne 0$. 
By Algorithm~\ref{alg:gor},
it follows that $\phi=b_d^*$ is a generator of~$\omega_R$ such 
that $\ord_{\mathcal{G}}(\phi)=-\ri(R)$.

Finally, we prove that~(b) implies~(a). The assumption implies that 
$\det(C_0) \ne 0$. Thus the claim follows from Algorithm~\ref{alg:GorCB}.
\end{proof}

The following example demonstrates that in Condition~(b) of this corollary 
the assumption $\ord_{\mathcal{G}}(\phi) = -\ri(R)$ is essential, 
even if~$R$ is a local ring.

\begin{example}\label{MaxOrdEssential}
Let $K=\mathbb{Q}$, let $P=K[x,y]$, let~$I$ be the ideal
$I = \langle  x^2,\; xy,\: y^3-x \rangle$ in~$P$, and let $R=P/I$.
Clearly, the ideal~$I$ is a primary ideal 
whose radical is $\langle x,\; y\rangle$.
The affine Hilbert function of~$R$ is $(1,3,4,4,\dots)$. Hence we have
$\ri(R)=2$ and $\Delta_R=1$.
A degree filtered $K$-basis~$B$ of~$R$ is given by the residue
classes of the elements in the tuple $(1,\; y,\; x,\; y^2)$.
In this case we have 
$$
C = \begin{pmatrix}
z_1 & z_2 & z_3 & z_4\\
z_2 & z_4 & 0 & z_3\\
z_3 & 0 & 0 & 0\cr
z_4 & z_3 &0 &0
\end{pmatrix}
$$
in Algorithm~\ref{alg:gor}, and thus $\det(C) = z_3^4$.
To get an element~$\phi$ in~$\omega_R$ of the form 
$\phi = c_1+c_2y^*+c_3x^*+c_4(xy)^*$
which generates this module, we therefore need $c_3\ne 0$.
But then $\phi(x)\ne 0$ implies $\phi(F_1R) \ne 0$,
and hence $\ord_\mathcal{G}(\phi )\ge -1 > -\ri(R)$.
In conclusion, the ring~$R$ is an example of a 0-dimensional
local affine $K$-algebra which is locally Gorenstein and satisfies 
$\Delta_R=1$, but does not have the~CBP.

Notice that in this example we have 
${\rm \id}_R^* = x\cdot x^*$, $y^* = y^2\cdot x^*$, and
$(y^2)^* = y\cdot x^*$. Therefore the element~$x^*$ generates~$\omega_R$.
On the other hand, we have $x\cdot (y^2)^* = 0$. Hence the element
$(y^2)^*$ of order~$-2$ has a non-trivial annihilator. By Condition~(c) 
of Theorem~\ref{CharGor}, this shows 
that~$(y^2)^*$ does not generate~$\omega_R$.
\end{example}

A small modification of the above example shows that there are 
0-dimensional Gorenstein local rings~$R$ 
with $\Delta_R > 1$. In other words, even if~$R$ is a local
Gorenstein ring, the homogeneous component of highest degree
of $\grFR$ need not be a 1-dimensional $K$-vector space.

\begin{example}\label{GorwithDeltabig}
Let $K=\mathbb{Q}$, let $P=K[x,y]$, let~$I$ be the ideal
$I = \langle  x^2,\; xy,\: y^2-x \rangle$ in~$P$, and let $R=P/I$.
Clearly, the ideal~$I$ is a primary ideal whose radical is 
$\langle x,\; y\rangle$. The affine Hilbert function of~$R$ is $(1,3,3,\dots)$.
Hence we have $\ri(R) = 1$ and $\Delta_R=2$.
A degree filtered $K$-basis~$B$ of~$R$ is given by the residue
classes of the elements in the tuple $(1,\; y,\; x)$.
In this case we have 
$$
C = \begin{pmatrix}
z_1 & z_2 & z_3 \\
z_2 & z_3 & 0 \\
z_3 & 0 & 0 
\end{pmatrix}
$$
in Algorithm~\ref{alg:gor}, and thus $\det(C) = -z_3^3$.
Consequently, the element $\phi=x^\ast$ generates~$\omega_R$.
Thus~$R$ is a 0-dimensional local affine $K$-algebra
which is a Gorenstein ring and satisfies $\Delta_R=2$.
\end{example}

\bigbreak
%
%

\section{Strict Gorenstein Rings and the Cayley-Bacharach Property}
\label{sec6}

In this section we let $R=P/I$ be a 0-dimensional affine 
$K$-algebra as above, and we let $\mathcal{F}=(F_i R)_{i\in\mathbb{Z}}$
be the degree filtration of~$R$. Recall that the graded ring of~$R$ satisfies
$$
\grFR \;=\; {\textstyle\bigoplus\limits_{i\in\mathbb{Z}}}
\; F_i R / F_{i-1}R \;\cong\; P / \DF(I)
$$
where $\DF(I)$ is the degree form ideal of~$I$ (see~\cite{KR2}, Example~6.5.11).
Consequently, in the following we identify the elements 
of $\grFR$with the corresponding elements in~$P/\DF(I)$.

\begin{remark}\label{hilbandAffhilb}
From the definition it follows that  $\HF^a_R(i) = \HF^a_{\grFR}(i)$ and 
$\HF_{\grFR}(i) = \Delta \HF^a_R(i)$ for all $i\in\mathbb{Z}$.
\end{remark}

The degree filtration and the graded ring are related to an embedding of
the scheme $\mathbb{X} =\Spec(P/I)$ into the projective $n$-space as follows.

\begin{remark}
Recall also that the homogenization of~$R$
is $R^\hom=\overline{P}/I^{\hom}$ where $\overline{P}=K[x_0,x_1,\dots,x_n]$
and where~$I^{\hom}$ is the homogenization of~$I$ with respect to~$x_0$.
Geometrically, the ring $R^\hom$ is the homogeneous coordinate
ring of the 0-dimensional scheme obtained by embedding
$\mathbb{X}=\Spec(P/I) \subset \mathbb{A}^n$
into projective $n$-space via $\mathbb{A}^n \cong D_+(x_0)
\subset \mathbb{P}^n$. The~CBP of~$R$
(or of~$\mathbb{X}$) can be reformulated as a property 
of~$R^\hom$ in a straightforward way. Since we focus on affine
schemes for the reasons explained in the introduction, we leave
this task to the interested reader.

In this setting, we have $\grFR \cong 
R^\hom/\langle x_0\rangle$ and $R\cong R^\hom/\langle x_0-1\rangle$,
where both~$x_0$ and $x_0-1$ are non-zerodivisors of~$R^\hom$.
Since both $R^{\hom}$ and $\grFR$
are standard graded $K$-algebras and since $x_0\in R^{\hom}$ is a homogeneous
non-zerodivisor, the ring $R^{\hom}$ is a Gorenstein ring if and only
if $\grFR$ is a Gorenstein ring.
On the other hand, the non-zerodivisor $x_0-1\in R^{\hom}$ is not
homogeneous. Thus the condition that $R^{\hom}$ is a Gorenstein ring implies
that $R \cong R^{\hom} / \langle x_0-1\rangle$ is locally Gorenstein,
but the converse is not true in general.
\end{remark}

In view of this remark, we introduce the following definition
(see also~\cite{KK}, Definition~3.2).

\begin{definition}\label{defstrictGor}
A 0-dimensional affine $K$-algebra $R=P/I$ is called a
{\bf strict Gorenstein ring} if $\grFR$
is a local Gorenstein ring. 
\end{definition}

By the preceding remark, strict Gorenstein rings are locally
Gorenstein, but the converse in not true in general, 
as Example~\ref{GorwithDeltabig} shows.
In~\cite{DGO} and~\cite{Kre1}, the property of~$R$
to be a strict Gorenstein ring was characterized in some special cases
by the~CBP and the symmetry of its Hilbert function. This
last property is defined as follows.

\begin{definition}\label{SymmHF}
The affine Hilbert function~$\HF^a_R$ of~$R$ is called {\bf symmetric} 
if and only if the Castelnuovo function $\Delta\HF^a_R$ satisfies
$\Delta\HF^a_R(\ri(R)-i)=\Delta\HF^a_R(i)$ for all $i\in\mathbb{Z}$.
\end{definition}

Alternatively, we can express the symmetry of~$\HF^a_R$ as follows.

\begin{proposition}\label{SymmHFChar}
For a 0-dimensional affine $K$-algebra~$R$, the following
conditions are equivalent.
\begin{enumerate}
\item[(a)] The affine Hilbert function of~$R$ is symmetric.

\item[(b)] For all $i\in\mathbb{Z}$, we
have $\dim_K(R)-\HF^a_R(\ri(R)-i) = \HF^a_R(i-1)$.

\end{enumerate}
\end{proposition}

\begin{proof}
To prove that~(a) implies~(b), we let $i\in\mathbb{Z}$
and calculate
$$
\dim_K(R) - \HF^a_R(\ri(R)-i) \;=
   {\textstyle\sum\limits_{j=\ri(R)-i+1}^{\ri(R)}} \Delta \HF^a_R(j)
   \;=\; {\textstyle\sum\limits_{j=0}^{i-1}} \Delta\HF^a_R(j) 
   \;=\; \HF^a_R(i-1) 
$$
Conversely, we let $i\in\mathbb{Z}$ and conclude from
\begin{align*}
\Delta \HF^a_R(\ri(R)-i) &\;=\; \HF^a_R(\ri(R)-i) - \HF^a_R(\ri(R)-i-1)\\
&\;=\; - \HF^a_R(i-1) + \HF^a_R(i) \;=\; \Delta \HF^a_R(i)
\end{align*}
that $\HF^a_R$ is symmetric.
\end{proof}

Notice that Condition~(b) is equivalent to the fact that
all inequalities in Corollary~\ref{HFinequal} are equalities.

\smallskip
In the sequel we identify the graded module of~$\omega_R$
with respect to the filtration~$\mathcal{G}$ with the 
canonical module of~$\grFR$ via the next lemma.
Recall that the leading form of an element 
$\phi\in\omega_R\setminus \{0\}$ of order~$\gamma$ is 
defined by $\LF_{\mathcal{G}}(\phi)= \phi+ G_{\gamma-1}\omega_R$.

\begin{lemma}\label{IsomOfgr}
The map $\Psi:\; \gr_{\mathcal{G}}\omega_R \longrightarrow
\omega_{\grFR} \cong \Hom_K(\grFR,K)$ defined by
$$
\Psi(\LF_{\mathcal{G}}(\phi))= \left( \LF_{\mathcal{F}}(f)\mapsto \begin{cases}
\phi(f) & \hbox{\it if } \ord_{\mathcal{F}}(f)=-\ord_{\mathcal{G}}(\phi), \\
0 & \hbox{\it otherwise,}
\end{cases} \right)
$$
is an isomorphism of $\grFR$-modules.
\end{lemma}

\begin{proof}
This is a consequence of a more general result given
in~\cite{NO}, Lemma I.6.4.
\end{proof}

To prove the theorem below, we need a further 
auxiliary result.

\begin{lemma}\label{AnnIngr}
Let $\phi\in\omega_R\setminus \{0\}$ be an element such that
$\Ann_{\grFR}(\LF_{\mathcal{G}}(\phi))=\{0\}$.
Then we have $\Ann_R(\phi)=\{0\}$.
\end{lemma}

\begin{proof}
For a contradiction, assume that there exists an element
$f\in R\setminus\{0\}$ such that $f\,\phi=0$.
Then we have $\LF_{\mathcal{F}}(f)\ne 0$ and
$\LF_{\mathcal{G}}(\phi)\ne 0$. By the hypothesis, this
implies $\LF_{\mathcal{F}}(f)\,\LF_{\mathcal{G}}(\phi)\ne 0$. 
Hence we have $\LF_{\mathcal{G}}(f\phi)=
\LF_{\mathcal{F}}(f)\,\LF_{\mathcal{G}}(\phi)\ne 0$, a contradiction.
\end{proof}

Now we are ready to characterize 0-dimensional strict Gorenstein rings
as follows.

\begin{theorem}{\bf (First Characterization of Strict Gorenstein 
Rings)}\label{CharSGor1}\\
Let $R=P/I$ be a 0-dimensional affine $K$-algebra.
Then the following conditions are equivalent.
\begin{enumerate}
\item[(a)] The ring~$R$ is a strict Gorenstein ring. 

\item[(b)] The ring~$R$ has the~CBP and a symmetric affine Hilbert function. 
\end{enumerate}
If these conditions are satisfied, then~$R$ is locally Gorenstein.
\end{theorem}

\begin{proof}
First we show that~(a) implies~(b). By the hypothesis, the
module $\omega_{\grFR}$ is a graded free $\grFR$-module
of rank one with a basis element~$\psi$ in degree $-\ri(R)$, i.e.,
we have $\omega_{\grFR}\cong \grFR(\ri(R))$.
Using the identification of Lemma~\ref{IsomOfgr}, we 
obtain an element $\phi\in\omega_R$ with $\ord_{\mathcal{G}}(\phi)=
-\ri(R)$ such that $\LF_{\mathcal{G}}(\phi)=\psi$.
Then Lemma~\ref{AnnIngr} yields $\Ann_R(\phi)=\{0\}$,
and therefore~$\phi$ is an $R$-basis of~$\omega_R$.
By Theorem~\ref{CharGor}, the ring~$R$ is locally Gorenstein.

Furthermore, using $\omega_{\grFR}\cong \grFR(\ri(R))$
and Remark~\ref{hilbandAffhilb}, we get the equalities 
$\HF^a_{\omega_{\grFR}}(-i)=
\HF^a_{\grFR}(\ri(R)-i) = \HF^a_R(\ri(R)-i)$ for all $i\in\mathbb{Z}$.
Now an application of Proposition~\ref{HFomegaprops}.d to
$\grFR$ yields 
$$
\HF^a_{\omega_{\grFR}}(-i) \;=\; \dim_K(\grFR) - 
\HF^a_{\grFR}(i-1) \;=\; \dim_K(R) - \HF^a_R(i-1)
$$
for all $i\in\mathbb{Z}$. Altogether, we get
$\HF^a_R(\ri(R)-i) = \dim_K(R) - \HF^a_R(i-1)$, and therefore
$\HF^a_R$ is symmetric by Proposition~\ref{SymmHFChar}. 
In particular, we 
have $\Delta_R=1$ and Corollary~\ref{deltaqual1} shows that~$R$
has the~CBP.

Now we show that~(b) implies~(a). By the symmetry of~$\HF^a_R$,
we have $\Delta_R=1$. Using Corollary~\ref{deltaqual1},
we find an element $\phi\in G_{-\ri(R)}\omega_R$
which is an $R$-basis of~$\omega_R$. 
Let~$B$ be a degree filtered $K$-basis
of~$R$. For every $i\ge 0$, we let $B_{\le i}$ be the subtuple
of~$B$ consisting of its elements of order $\le i$.
Since $\Ann_R(\phi)=\{0\}$, the elements in $B_{\le i}\cdot\phi$
are $K$-linearly independent. They generate a $K$-vector subspace~$V_i$
of dimension $\HF^a_R(i)$ of~$G_{-\ri(R)+i}\omega_R$. Using
Proposition~\ref{HFomegaprops}.d and the symmetry of~$\HF^a_R$,
we get $\dim_K G_{-\ri(R)+i}\omega_R = \dim_K(R) - \HF^a_R(\ri(R)-i-1) = 
\HF^a_R(i)$. Hence we have $V_i=G_{-\ri(R)+i}\omega_R$ for every $i\ge 0$.

In order to prove that $\grFR$ is a Gorenstein ring,
it suffices to show that the element $\LF_{\mathcal{G}}(\phi)$
is a basis of~$\omega_{\grFR}$. For this purpose we need to prove
that its annihilator is zero. For a contradiction, assume
that there exists an element $f\in R\setminus \{0\}$
such that $\LF_{\mathcal{F}}(f)\LF_{\mathcal{G}}(\phi)=0$.
Letting $i=\ord_{\mathcal{F}}(f)$, we obtain
$\ord_{\mathcal{G}}(f\phi)\le -\ri(R)+i-1$. Hence the
element $f\phi \in G_{-\ri(R)+i-1}\omega_R = V_{i-1}$ 
is a linear combination of the elements in
$B_{\le i-1}\cdot\phi$. Now the condition $\Ann_R(\phi)=\{0\}$
implies $f\in \langle B_{\le i-1}\rangle_K$, and therefore
$\ord_{\mathcal{F}}(f)\le i-1$, a contradiction.
\end{proof}

Of course, since the graded ring $\grFR$ is a 0-dimensional
affine $K$-algebra, too, we can examine whether it has the~CBP
or not. For this purpose, we introduce the following notion.

\begin{definition}
The 0-dimensional affine $K$-algebra $R=P/I$ is said to have
the {\bf strict Cayley-Bacharach property} if its graded
ring $\grFR$ has the Cayley-Bacharach property. 
\end{definition}

Our next proposition provides a justification for this terminology.

\begin{proposition}\label{SCBPimpliesCBP}
Let $R=P/I$ be a 0-dimensional affine $K$-algebra.
If $R$ has the strict Cayley-Bacharach property, then~$R$
has the Cayley-Bacharach property.
\end{proposition}

\begin{proof}
To prove this implication, we use the characterization
of the CBP given in Theorem~\ref{CanModCB}. For a contradiction,
assume that there exists an element $f\in R\setminus \{0\}$ 
which satisfies $f\cdot G_{-\ri(R)}\omega_R =0$. Let $\phi
\in G_{-\ri(R)}\omega_R\setminus \{0\}$, and let $\LF_\mathcal{G}(\phi)$
be the leading form of~$\phi$, i.e.\ the residue class of~$\phi$
in 
$$
\gr_{\mathcal{G}}(\omega_R)_{-\ri(R)} \;=\; G_{-\ri(R)}\omega_R / 
G_{-\ri(R)-1} \omega_R \;=\; G_{-\ri(R)}\omega_R / \langle 0\rangle
$$
If we have $\LF_{\mathcal{F}}(f)\cdot \LF_{\mathcal{G}}(\phi)\ne 0$,
then this element equals $\LF_{\mathcal{G}}(f\phi)$, in contradiction
to $f\phi=0$. It follows that the leading form $\LF_{\mathcal{F}}(f)$ 
annihilates all elements in $\gr_{\mathcal{G}}(\omega_R)_{-\ri(R)}$.
Using Lemma~\ref{IsomOfgr}, we then get that
$\LF_{\mathcal{F}}(f)$ annihilates all elements in $(\omega_{\grFR})_{-\ri(R)}$.
In view of Theorem~\ref{CanModCB}, this contradicts the~CBP of~$\grFR$.
\end{proof}

The next example shows that the converse of this proposition
does not hold in general.

\begin{example}\label{7nonreduc}
Let $K=\mathbb{Q}$, let $P=K[x,y]$, let
$\M_1 = \langle y-x^2,\; x^3-x-1 \rangle$, let 
$\Q_2 = \langle x^2,\; y^2 \rangle$,
let $I$ be the ideal 
$$
I \;=\; \M_1\cap \Q_2 \;=\; 
\langle  xy^2 -y^3 -x^2 +y^2,\; x^2y -y^2,\;  x^3 -y^3 +y^2,\;  
y^4 -2y^3 -x^2 +y^2\rangle
$$
and let $R=P/I$.
The affine Hilbert function of~$R$ is $(1,3,6,7,7,\dots)$. 
Hence we have $\ri(R)=3$ and $\Delta_R=1$.

First we use Algorithm~\ref{alg:GorCB} to show that~$R$
is a locally Gorenstein ring having the~CBP.
A degree filtered $K$-basis~$B$ of~$R$ is given by the residue
classes of the elements in the tuple
$(1,\;  y,\;  x,\;  y^2,\;  xy,\;  x^2,\;  y^3)$.  
The computation of the matrix~$C_0$ in Step~(3) of Algorithm~\ref{alg:GorCB}
yields $\det(C_0) = z_7^7\ne 0$. This implies the claim. 
More precisely, by Corollary~\ref{deltaqual1} we know that
the element $\phi =(y^3)^*$ is a generator of~$\omega_R$ of order
$\ord_{\mathcal{G}}(\phi) = -\ri(R)$, and up to scalar multiples it is the
only element of that order.

Now we check that~$R$ does not have the strict CBP.
The degree form ideal of~$I$ is $\DF(I)= \langle xy^2 -y^3,\; x^2y,\; 
x^3-y^3,\; y^4 \rangle$. When we use Algorithm~\ref{alg:GorCB}
to check whether~$\grFR$ has the CBP, we find $\det(C_0)=0$.
Hence the ring $\grFR \cong P/\DF(I)$ does not have the CBP
and~$R$ does not have the strict CBP, as claimed.
\end{example}

Using the strict Cayley-Bacharach property, we can 
characterize strict Gorenstein rings in another way.

\begin{theorem}{\bf (Second Characterization of Strict Gorenstein
Rings)}\label{CharSGor2}\\
For a 0-dimensional affine $K$-algebra $R=P/I$, the following
conditions are equivalent.
\begin{enumerate}
\item[(a)] The ring $R$ is a strict Gorenstein ring.

\item[(b)] The ring~$R$ has the strict Cayley-Bacharach property and
satisfies $\Delta_R=1$.
\end{enumerate}
\end{theorem}

\begin{proof} 
To prove that~(a) implies~(b), we note that, since
$\grFR$ is a 0-dimen\-sional graded local Gorenstein ring,
its canonical module is a graded free $\grFR$-module
of rank one. More precisely, we have $\omega_{\grFR} \cong
\grFR(\ri(R))$, and thus the claim follows by applying 
Corollary~\ref{deltaqual1} to~$\grFR$.
The converse implication follows immediately from an application
of Corollary~\ref{deltaqual1} to~$\grFR$.
\end{proof}

Finally, we recall from~\cite{KK}, Proposition~4.2, that the ring~$R$ is said to be
a {\bf strict complete intersection} if $I=\langle f_1,\dots,f_n\rangle$
is generated by a regular sequence $(f_1,\dots,f_n)$ such that also
the degree forms $(\DF(f_1),\dots,\DF(f_n))$ form a regular sequence.
Strict complete intersections are strict Gorenstein rings.
Our last example brings us back to the origins of the history of
the CBP: it showcases a strict complete intersection of two
irreducible plane curves which has the CBP. However, we point
out that this ideal has non-maximal primary components and
one of its maximal components is non-linear. In other words, the corresponding
0-dimensional scheme is not reduced and has non-rational support.

\begin{example}\label{twocubics}
Let $K=\mathbb{Q}$, let $P=K[x,y]$, let~$I$ be the ideal
$I = \langle f,\, g\rangle$ in~$P$, where $f = y^3-x$ and 
$g= x^3 -3x^2y -xy^2 -2x^2 -xy +5y^2 +3x +2y$, and let $R=P/I$.
The primary decomposition of~$I$ is given by
$I = \M_1\cap \M_2\cap \Q_3 \cap \M_4$, where
$\M_1 = \langle x,\; y\rangle$,
$\M_2 =\langle x-8,\; y-2 \rangle$,
$\Q_3 =\langle x -3y-2,\; (y+1)^2\rangle$,
and $\M_4 =\langle x-y^3,\; y^5-y-1\rangle$.
The affine Hilbert function of~$R$ is $(1,3,6,8,9,9,\dots)$.
Hence we have $\ri(R)=4$ and $\Delta_R=1$.
Moreover, the affine Hilbert function of~$R$ is clearly symmetric.

Let us use Algorithm~\ref{alg:GorCB} to check that~$R$
is locally Gorenstein and has the~CBP.
A degree filtered $K$-basis~$B$ of~$R$ is given by the residue
classes of the elements in the tuple
$(1,\;  y,\;  x,\;  y^2,\;  xy,\;  x^2,\;  xy^2,\;  x^2y,\;  x^2y^2)$. 
The computation of the matrix~$C_0$ in Step~(3) yields
$\det(C_0) = z_9^9\ne 0$, and therefore the claim.
More precisely, the element $\phi =(x^2y^2)^*$ is a generator of~$\omega_R$ 
such that $\ord_{\mathcal{G}}(\phi) = -\ri(R)$. 

Altogether, Theorem~\ref{CharSGor1} shows that~$R$ is a strict
Gorenstein ring. In fact, since $(\DF(f),\, \DF(g))=(y^3,\,
x^3-3x^2y-xy^2)$ is a regular sequence, the ring~$R$ is a
strict complete intersection.
\end{example}

\subsection*{Acknowledgements} The third author thanks the University of
Passau for its hospitality and support during part of the preparation
of this paper.

\bigbreak


\begin{thebibliography}{99}

\bibitem{CoCoA} J.~Abbott, A.M.~Bigatti, and L.~ Robbiano, 
\cocoa : a system for doing Computations in Commutative Algebra, 
available at {\tt http://cocoa.dima.unige.it}.


\bibitem{Bac} I.~ Bacharach, Ueber den Cayley'schen 
Schnittpunktsatz, Math.\ Ann.\ {\bf 26} (1886), 275-299.

\bibitem{Cay1} A.\ Cayley, On the intersection of curves,
Cambridge Math.\ J.\ {\bf 3} (1843), 211-213.

\bibitem{Cay2} A.\ Cayley, On the intersection of curves,
Math.\ Ann.\ {\bf 30} (1887), 85-90.

\bibitem{Cha} M.\ Chasles, {\it Apercu historique sur
l'origine et le d\'eveloppement des m\'ethodes en 
g\'eometri\'e}, M. Hayez, Brussels, 1837.


\bibitem{DGO} E.~Davis, A.V.~Geramita, and F.~Orecchia, 
Gorenstein algebras and the Cayley-Bacharach theorem,
Proc. Amer. Math. Soc. {\bf 93} (1985), 593-597.

\bibitem{EGH} D.\ Eisenbud, M.\ Green, and J.\ Harris,
Cayley-Bacharach theorems and conjectures,
Bull.\ Amer.\ Math.\ Soc.\ {\bf 33} (1996), 295-324.

\bibitem{Eul} L.\ Euler, Sur une contradiction apparente
dans la doctrine des lignes courbes, Memoires de l'academie 
des sciences de Berlin {\bf 4} (1748), 219-233.

\bibitem{GKR} A.V. Geramita, M. Kreuzer, and L. Robbiano,
Cayley-Bacharach schemes and their canonical modules,
Trans.\ Amer.\ Math.\ Soc.\ {\bf 339} (1993), 163-189.

\bibitem{Gor} D.~Gorenstein, An arithmetic theory of adjoint plane
curves, Trans. Amer. Math. Soc. {\bf 72} (1952), 414-436.

\bibitem{Jac1} C.G.~Jacobi, Theoremata nova algebraica circa
systema duarum aequationum inter duas variabiles propositarum,
J.\ reine angew.\ Math.\ {\bf 14} (1835), 281-288.

\bibitem{Jac2} C.G.~Jacobi, De relationibus, quae locum habere 
debent inter pucta intersectionis duarum curvarum vel trium
superficierum algebraicarum dati ordinis, simul cum enodatione
paradoxi algbraici, J.\ reine angew.\ Math.\ {\bf 15} (1836), 205-308.

\bibitem{Kre1} M. Kreuzer, On 0-dimensional complete intersections,
Math.\ Ann.\ {\bf 292} (1992), 43-58.

\bibitem{Kre2} M. Kreuzer,
On the canonical module of a 0-dimensional scheme,
Can.\ J.\ Math. {\bf 141} (1994), 357-379.

\bibitem{KK} M.\ Kreuzer and E.\ Kunz, Traces in strict Frobenius
algebras and strict complete intersections, J.\ reine angew.\ Math.\ 
{\bf 381} (1987), 181-204.

\bibitem{KL} M.\ Kreuzer and L.N.~Long,
Characterizations of zero-dimensional complete intersections,
Beitr\"{a}ge Algebra Geom.\ {\bf 58} (2017), 93--129.

\bibitem{KLR} M.\ Kreuzer, L.N.\ Long, and L.\ Robbiano,
Subschemes of border basis schemes, in preparation.

\bibitem{KR1} M.\ Kreuzer and L.\ Robbiano, {\it Computational
Commutative  Algebra 1}, Springer Verlag, Heidelberg, 2000 (second printing 2008).

\bibitem{KR2} M.\ Kreuzer and L.\ Robbiano, {\it Computational
Commutative Algebra 2}, Springer Verlag, Heidelberg, 2005.

\bibitem{KR3} M.\ Kreuzer and L.\ Robbiano, {\it Computational  Linear
and Commutative Algebra}, Springer Int.\ Publ., Cham, 2016.

\bibitem{Lon} L.N.~Long,
Various differents for 0-dimensional schemes and applications,
dissertation, University of Passau, Passau, 2015.

\bibitem{NO} C.\ Nastasescu and F.\ van Oystaeyen, {\it Graded
and Filtered Rings and Modules}, Lecture Notes in Math.\ 
{\bf 758}, Springer Verlag, Heidelberg, 1979. 

\bibitem{Pap} Pappos of Alexandria, Mathematicae Collectiones, published by
F. de F.~Senense, Venice, 1588.

\bibitem{Pas} B.~Pascal, Essay pour les coniques, Paris, 1640.

\end{thebibliography}
\end{document}